\documentclass[12pt]{amsart}

\usepackage{amssymb}
\usepackage{cite}
\usepackage{url}
\usepackage{pgfplots}
\pgfplotsset{compat=newest}
\usepgfplotslibrary{polar}
\usepackage{braket}
\usepackage[export]{adjustbox}
\usepackage{subfig}
\usepackage{booktabs}
\usepackage{siunitx}
\usepackage{enumitem}
\usepackage{color}
\usepackage{tikz}
\numberwithin{equation}{section}

\usepackage{changes}
\definechangesauthor[name=Daniele, color=blue]{dp}

\newcommand{\roundPrecision}{2}
\DeclareMathOperator*{\argmin}{arg\,min}

\renewcommand{\xi}{\gamma}
\newcommand{\up}{u_\pi}
\newcommand{\uI}{u_I}
\newcommand{\ff}{\rm f}
\newcommand{\PP}{E}
\renewcommand{\P}{\mathbb P}

\newtheorem{theorem}{Theorem}[section]
\newtheorem{lemma}[theorem]{Lemma}
\newtheorem{remark}[theorem]{Remark}
\newtheorem{corollary}[theorem]{Corollary}

\thanks{{DP has been partially funded by the project DIT.AD021.185.001 - Metodi analitici e numerici per equazioni differenziali. All the authors are also members of the Gruppo Nazionale Calcolo Scientifico-Istituto Nazionale di Alta Matematica (GNCS-INdAM)}}

\begin{document}

\markboth{Daniele Prada, Franco Brezzi, L. Donatella Marini}{A Virtual Element Method on polyhedra with curved faces}

%
%




\title[A Virtual Element Method on polyhedra with curved faces]{A Virtual Element Method\\ on polyhedra with curved faces}


\author[D. Prada]{Daniele Prada} 

\address{IMATI del CNR, Pavia, (Italy)}
\email{prada@imati.cnr.it}

\author[F. Brezzi]{Franco Brezzi}
\address{IUSS, and IMATI del CNR, Pavia (Italy)}
\email{brezzi@imati.cnr.it}

\author[L. D. Marini]{L. Donatella Marini}
\address{Dipartimento di Matematica, Universit\`a di Pavia,
and IMATI del CNR, Pavia, (Italy)}
\email{marini@imati.cnr.it}

\maketitle


\begin{abstract}
  In this paper we construct conforming Virtual Element approximations on domains with curved boundary and/or internal curved interfaces, both in two and three dimensions.
  Our approach allows to impose both Dirichlet and Neumann non-homogeneous boundary conditions, and provides, for degree of accuracy $k \geq 1$, optimal convergence rates.
  Whenever the exact solution is a polynomial of degree $k$, local spaces of degree $k$ ensure satisfaction of the patch test.
  The proposed method is theoretically analyzed in the two-dimensional case, whereas it is numerically validated both in two and three dimensions.
\end{abstract}











\section{Introduction}

The Virtual Element Method (VEM) can be viewed as an extension of the classical Finite Element Method (FEM) to general polytopal tessellations. Since their introduction in \cite{basic_vem}, Virtual Elements have shown an increasing interest for their ductility and the possibility to create conforming approximations on very general polygonal/polyhedral decompositions. The method has been applied to many different applications. It would be difficult, if not impossible, to give credit to all the papers appeared during the years. We refer to \cite{Acta-VEM} for an (almost) exhaustive bibliography. Not much however has been done for dealing with curved edges/surfaces. Nonconforming approximations, as in nonconforming VEM and HHO  methods (see \cite{nonconforming, HHO-curved}), are relatively easy, since the traces of functions do not need to be continuous at the curved interfaces. Conforming VEM on curved edges, given in parametric form, were first introduced in \cite{beirao_curvo_2019}, using, on the curved edges, functions that are polynomials in the parameter. The approach is simple and effective, but the VEM local spaces do not contain polynomials of a given degree, so that the {\it patch test} is not satisfied. Another approach was presented in  \cite{beirao2020}, with the advantage that the local spaces contain all polynomials of a given degree, and the result does not depend on the choice of parametrization. The disadvantage is the need of introducing additional degrees of freedom and the difficulty of programming the method. A similar approach was presented independently in \cite{anand2020}.
VEM have also been combined with boundary correction techniques, such as the ones underlying the shifted boundary method~\cite{atallah2020}, to deal with curved boundaries~\cite{bertoluzza2022, bertoluzza2024}. The extension of the previous technologies to curved internal interfaces and Neumann boundary conditions is still an area of active research.

We propose a conforming approach for two and three dimensional problems, which allows to handle internal curved interfaces and mixed Dirichlet/Neumann non-homogeneous boundary conditions, and ensures optimal convergence rates. Moreover, whenever the exact solution is a polynomial of degree $k$, the solution of the discretized problem of order $k$ will coincide with the exact solution, thus ensuring the patch test. 
We use the Poisson problem as an example in order to explain the construction of the numerical method. Surely, it would be particularly interesting to extend it to more complex problems, such as: linear elasticity problems; incompressible and nearly incompressible materials, contact problems, electromagnetic problems, mesh adaptation for dealing with singularities, and so forth.

The paper is organized as follows. In Sec.~\ref{sec:continuous} we present the model problem. In Sec.~\ref{sec:localVEM2d}, the main part of the paper, we introduce the two-dimensional local spaces to be used in a (possibly curved) polygonal mesh element. In Sec.~\ref{sec:discrete2d} we write the discretized problem and provide an error estimate in the $H^1$ norm. The extension to the three-dimensional case is presented in Sec.~\ref{sec:VEM3d}. Finally, in Sec.~\ref{sec:numerical} we present some numerical tests, both in two and three dimensions.

\section{The Continuous Problem}\label{sec:continuous}

\subsection{Notation}

\noindent
For a nonnegative integer $k$, $\mathbb P_{k,d}$ is the space of polynomials of degree $\leq k$ in $d$ dimensions, $d = 2, 3$. By convention, $\mathbb P_{-1,d} = \{0\}$. Given $\mathcal O \subset \mathbb R^d$, $\mathbb P_k(\mathcal O)$ is the restriction of $\mathbb P_{k,d}$ to $\mathcal O$. We denote by
\begin{equation}\label{poly-hom}
{\mathbb P}^{\rm hom}_s(\mathcal O):=\{\mbox{homogeneous polynomials of degree}\, s \mbox{ in the variable } (x_i-x_{i,\mathcal O})/h_{\mathcal O} \},
\end{equation}
where $x_{i,\mathcal O}$ are the coordinates of the centroid of $\mathcal O$ and $h_{\mathcal O}$ is the diameter of $\mathcal O$.

\subsection{The continuous problem}

\noindent
Let $\Omega$ be a bounded open subset of $\mathbb R^d, d = 2,3$, with Lipschitz continuous boundary $\Gamma$. We assume $\Gamma = \overline\Gamma_D\cup\overline\Gamma_N$, with $|\Gamma_D| > 0$, $\mathring\Gamma_D\cap\mathring\Gamma_N = \emptyset$. 

We consider the following model problem
\begin{equation}\label{eq:model}
\begin{cases}
\text{find } u \in H^1(\Omega) \text{ such that:}\\
-\nabla\cdot(\kappa\nabla u) = f\quad\text{in }\Omega,\\
u = g_D\text{ on }\Gamma_D, \quad (\kappa\nabla u)\cdot\mathbf n = g_N\text{ on }\Gamma_N,
\end{cases}
\end{equation}
where $\kappa$ is a positive piecewise constant function in $\Omega$,  
$f, g_D, g_N$ are given functions with, say, $f\in L^2(\Omega), g_D = H^{1/2}(\Gamma_D)$ and $g_N = L^2(\Gamma_N)$.

By defining, for $\phi\in H^{1/2}(\Gamma_D)$,
\[
H^1_{\phi,\Gamma_D}(\Omega) = \Set{v \in H^1(\Omega) | v = \phi \text{ on }\Gamma_D},
\]
the variational formulation of~\eqref{eq:model} reads
\begin{equation}\label{eq:model-weak}
\begin{cases}
\text{find } u \in H^1_{g_D,\Gamma_D}(\Omega) \text{ such that}\\
a(u,v) = (f,v) + \langle g_N, v\rangle\quad\forall\,v \in H^1_{0,\Gamma_D}(\Omega),
\end{cases}
\end{equation}
with
\begin{equation}\label{eq:model-weak-symbols}
a(u,v) = \int_\Omega\kappa\nabla u\cdot\nabla v\,dx,\quad(f,v) = \int_\Omega f\,v\,dx,\quad\langle g_N, v\rangle = \int_{\Gamma_N}g_N\,v\,d\Gamma.
\end{equation}


\section{Local Virtual Element Spaces in Two Dimensions}\label{sec:localVEM2d}

\noindent
Let $\PP$ be a bounded polygon with diameter $h_{\PP}$. Without loss of generality, we assume that $\PP$ has at most one edge that will be treated as ``curved'' (from a geometrical point of view, it could be either curved or straight or almost straight). All the other edges are assumed to be straight. This condition is mandatory for the approach followed in this work.

In order to define local virtual element spaces for the discretization of~\eqref{eq:model-weak}, we have to distinguish between different types of polygons:
\begin{itemize}
    \item Polygons with straight edges.
    \item Polygons with a curved edge shared with another polygon. In this case, we will follow a master and slave approach (only one element sets the degrees of freedom to be used on the curved edge).
    \item Polygons with a curved edge that belongs entirely to $\overline\Gamma_D$.
    \item Polygons with a curved edge that belongs entirely to $\overline\Gamma_N$.
\end{itemize}

\noindent
We assume that each polygon $\PP$ verifies the following conditions:
\begin{enumerate}[label=(A\arabic*)]
    \item\label{mesh2D-C1} the boundary of $\PP$ is piecewise $C^1$,
      \item\label{mesh2D-star} $\PP$ is star-shaped with respect to  every point of a disk $D$ of radius $\rho^{\PP} h_{\PP}$, with $\rho^{\PP} \geq\rho_0 > 0$,
     \item\label{mesh2D-Nedge} every edge $e$ of $\PP$ has lenght $|e|\ge \rho^{\PP} h_{\PP}$.
%
   \end{enumerate}
The above regularity assumptions can be relaxed~\cite{beirao2017,brenner2017,brenner2018,beirao_curvo_2019}.

We will use the notation $A \lesssim B$ to represent the inequality $A \leq (\text{constant})B$, where the positive constant depends only on $k$ (possibly on $\rho_0$) and that may change at each occurrence. The notation $A \approx B$ is equivalent to $A \lesssim B$ and $A \gtrsim B$.

\subsection{Subspaces on polygons with straight edges}\label{subsec:VkP2d_straight}

\noindent
Let us recall the local spaces as introduced in~\cite{ahmad2013} (see also \cite{Acta-VEM}). Given an integer $k \geq 1$, we let
\begin{equation}\label{eq:VkP2d_straight}
 {\widetilde V}^k(\PP) := \{v \in H^1(\PP) \cap C^0(\overline\PP) : v_{|e} \in \mathbb P_k(e)\,\forall\text{ edge }e\subset\partial\PP, \,\Delta v \in\mathbb P_{k}(\PP)\}.
\end{equation}
The following linear operators are well defined in ${\widetilde V}^k(\PP)$:
\begin{equation}\label{eq:VkP2d_straight_dofs}
    \begin{aligned}
    &\bullet\,\,\text{the values of $v$ at each vertex of $\PP$,}\\
    &\bullet\,\,\text{(for $k\geq 2$) the values of $v$ at the $k-1$ Gauss-Lobatto points}\\
    &\qquad \text{of each edge $e\subset\partial\PP$,}\\
    &\bullet\,\,\text{ (for $k \geq 2$) the moments ${\displaystyle\frac{1}{|\PP|}\int_{\PP}v p_{\alpha}\,dx}\quad\forall p_{\alpha}\in\mathbb P_{k-2}(\PP)$}.\\
    \end{aligned}
\end{equation}
Here, denoting by $ \alpha= (\alpha_1, \alpha_2)\in\mathbb N^2$ a multi-index, we use
\begin{equation}\label{eq:Pk2d}
p_\alpha(x,y) = \left(\frac{x-x_{\PP}}{h_{\PP}}\right)^{\alpha_1}\left(\frac{y-y_{\PP}}{h_{\PP}}\right)^{\alpha_2}\
\quad\alpha_1+\alpha_2 {\color{black}\leq k-2},
\end{equation}
where $(x_{\PP}, y_{\PP})$ is the centroid of $\PP$, and $h_\PP$ is the diameter of $\PP$. Note that $\|p_\alpha\|_{L^\infty(\PP)} \leq 1$.

\noindent
The operators \eqref{eq:VkP2d_straight_dofs} allow to compute useful projections onto polynomials. In particular,  the operator $\Pi^\nabla_{k,\PP}: \, H^1(\PP) \,\longrightarrow \, \mathbb P_k(\PP)$ defined by
    \begin{equation}\label{eq:PiNablaP}
\int_{\PP}\nabla(\Pi^\nabla_{k,\PP} v-v)\cdot \nabla q\, dx=0\quad \int_{\partial\PP}(\Pi^\nabla_{k,\PP} v-v)\, ds=0
\qquad\forall q\in \mathbb P_k(\PP)
\end{equation}
is computable, and it allows to define the final local spaces as:
\begin{equation}\label{eq:VkP2d_straight-final}
V^k(\PP)=\{ v \in {\widetilde V}^k(\PP) | \int_{\PP} v\,q_s \,dx = \int_{\PP}\Pi^\nabla_{k,\PP} v\,q_s \,dx\quad \forall q_s\in {\mathbb P}^{\rm hom}_s(\PP), \,s=k-1,k\}.
\end{equation}


\noindent
Other useful projections that can be computed in ${V}^k(\PP)$ are:
\begin{description}
 \item[\it $L^2$ projection.] $\Pi^0_{k,\PP}: L^2(\PP) \, \longrightarrow \,\mathbb P_k(\PP)$, defined by
        \begin{equation}\label{eq:Pi0P}
        \int_{\PP}\Pi^0_{k,\PP}v\,q\,dx = \int_{\PP}v\, q\,dx\quad\forall\,q\in\mathbb P_k(\PP).
    \end{equation}
  \item[\it $L^2$ projection of the gradient.] $\Pi^0_{k-1,\PP}:  H^1(\PP)  \longrightarrow [\mathbb P_{k-1}(\PP)]^2$, defined by
    \begin{equation}\label{eq:PiGradP}
        \int_{\PP}(\Pi^0_{k-1,\PP}\nabla v)\cdot\mathbf q\,dx = \int_{\PP}\nabla v\cdot\mathbf q\,dx\quad\forall\,\mathbf q\in[\mathbb P_{k-1}(\PP)]^2.
    \end{equation}
\end{description}
Clearly, for any $q\in\mathbb P_k(\PP)$, we have:
\begin{equation}\label{projections}
\Pi^\nabla_{k,\PP}\,q = q,\quad \Pi^0_{k-1,\PP}\nabla q = \nabla q, \quad \Pi^0_{k,\PP}\,q = q.
\end{equation}

\subsection{Subspaces on curved polygons declared as \emph{master}}
\label{subsec:VkP2d_master}

\noindent
Let $k \geq 1$ be a given integer. Given two polygons sharing a curved edge $\xi$, we arbitrarily choose one of them to set the degrees of freedom on $\xi$ and we name it \emph{master} polygon. Let $\PP$ be a given master polygon. In order to define a space of functions like~\eqref{eq:VkP2d_straight-final}, we need to define them on $\xi$. Let us start by considering the data we have (candidates to become degrees of freedom):
\begin{equation}\label{eq:VkP2d_master_dofs}
    \begin{aligned}
        &\bullet\,\,\text{the values of $v$ at the vertices of $\PP$},\\
        &\bullet\,\,\text{(for $k \geq 2$) the values of $v$ at the $k-1$ Gauss-Lobatto points}\\
        &\qquad\text{of each straight edge $e\subset\partial\PP$},\\
        &\bullet\,\,\text{(for $k \geq 2$) ${\displaystyle\frac{1}{|\PP|}\int_{\PP}v p_{k-2}\,dx}\quad\forall p_{k-2}\in\mathbb P_{k-2}(\PP)$}.
    \end{aligned}
\end{equation}
\begin{lemma}
If $\PP$ is a triangle-like element, i.e., with two straight edges and one curved, conditions \eqref{eq:VkP2d_master_dofs} identify a unique polynomial $p_k$ in $\mathbb P_k(\PP)$. 
\end{lemma}
\begin{proof}
The number of conditions equals the dimension of $\mathbb P_k(\PP)$. Indeed,
\[
3 + 2(k-1) + k(k-1)/2 = (k+1)(k+2)/2 \equiv \textit{dim}\,\mathbb P_k(\PP).
\]
Hence, we have to prove that, if $p_k$ has  vanishing conditions, then it is identically zero. The first two conditions imply that $p_k$ vanishes on the two straight edges. Therefore, if $\lambda_1, \lambda_2$ are the equations of the two edges,
$$
p_k= \lambda_1 \lambda_2 q_{k-2} \qquad \mbox{for some } q_{k-2} \in \mathbb P_{k-2}.
$$
From the third set of conditions we have
$$
\int_{\PP} \lambda_1 \lambda_2 (q_{k-2})^2 \,dx = 0,
$$
and since the product $ \lambda_1 \lambda_2$ does not change sign in $\PP$ we deduce $p_k\equiv 0$.
\end{proof}
With a given virtual function $v$, we can then associate a polynomial $p_k^*\in\mathbb P_k(\PP)$ computed through the conditions:
\begin{equation}\label{eq:lsq2d_tri}
    \begin{cases}
        p_k^* = v \text{ at the vertices of $\PP$},\\
        \text{(for $k \geq 2$)\quad $p_k^* = v$ at the $k-1$ Gauss-Lobatto points}\\
        \text{\qquad $\forall$\,straight edge $e\subset\partial\PP$},\\
        \text{(for $k \geq 2$)\quad${\displaystyle\int_{\PP}p_k^* p_{k-2}\,dx} =\displaystyle{\int_{\PP}v p_{k-2}\,dx}\quad\forall p_{k-2}\in\mathbb P_{k-2}(\PP)$}.
    \end{cases}
\end{equation}
Then we can take
\[
v_{|\xi} = {p_k^*}_{|\xi}.
\]
Note that, here, $v \equiv p_k^*$ on the whole element $\PP$. By construction, global continuity is guaranteed.

For a general polygon with a number of straight edges bigger than 2, the number of conditions~\eqref{eq:VkP2d_master_dofs} is bigger than the dimension of $\mathbb P_k(\PP)$. In this case we can use a least-square approach, as done in \cite{curved_bm},  keeping the values at the two endpoints of $\xi$ fixed to guarantee global continuity. Specifically, denoting by $N$ the
number of conditions~\eqref{eq:VkP2d_master_dofs}, ordered in such a way that the values at the two endpoints of $\xi$ are the last two, we look for $p_k^*\in\mathbb P_k(\PP)$ such that:
\begin{equation}\label{eq:lsq2d}
    \begin{cases}
        \displaystyle p_k^* = \argmin_{p_k\in\mathbb P_k(\PP)}\sum_{i=1}^{N-2}(\text{dof}_i(p_k)-\text{dof}_i(v))^2\\
        \text{and }p_k^* = v\text{ at the two endpoints of }\xi.
    \end{cases}
\end{equation}
Then, as before, we set
\begin{equation}\label{eq:lsq2d_sol}
    v_{|\xi} = {p_k^*}_{|\xi}.
\end{equation}

Once the function $v$ is individuated through conditions \eqref{eq:VkP2d_master_dofs} and \eqref{eq:lsq2d_sol}, we can compute $\Pi^\nabla_{k,\PP} v$. Mimicking what we did in Section \ref{subsec:VkP2d_straight} we then introduce the space
\begin{multline}\label{eq:VkP2d_master}
    {\widetilde V}^k(\PP) := \{\,v \in H^1(\PP)\cap\,C^0(\overline\PP) : v_{|e} \in \mathbb P_k(e)\,\,\forall\text{ straight edge }e\subset\partial\PP,\\
    v_{|\xi} = {p_k^*}_{|\xi},  \,\Delta v \in\mathbb P_{k}(\PP)\,\},
\end{multline}
and then
\begin{equation}\label{eq:VkP2d_master-final}
V^k(\PP)=\{ v \in {\widetilde V}^k(\PP) : \int_{\PP} v\, q_s \,dx = \int_{\PP}\Pi^\nabla_{k,\PP} v\, q_s \,dx\quad \forall q_s\in {\mathbb P}^{\rm hom}_s(\PP), \,s=k-1,k\}.
\end{equation}
We point out that, for every $k \geq 1$,
\begin{equation}\label{eq:patch-test-2d}
    \mathbb P_k(\PP) \subseteq V^k(\PP).
\end{equation}
\begin{remark}\label{alternative-procedure}
Another strategy could be used, more in the spirit of Mimetic Finite Differences: instead of defining the trace of $v$ on $\xi$, we compute its moments. More precisely,  we look for $p_k^*\in\mathbb P_k(\PP)$ such that:
\begin{equation}\label{eq:lsq2d-mim}
        \displaystyle p_k^* = \argmin_{p_k\in\mathbb P_k(\PP)}\sum_{i=1}^{N}(\text{dof}_i(p_k)-\text{dof}_i(v))^2.
\end{equation}
A function $v$ will be individuated by conditions \eqref{eq:VkP2d_master_dofs} plus the moments of 
$p_k^*$ of order up to  $k$ on $\xi$:
\begin{equation}\label{mimetic-2D}
\int_{\xi} v q_{k} \,d\xi = \int_{\xi}  p_k^* q_{k} \,d\xi \quad \forall q_{k} \in \mathbb P_{k}(\PP).
\end{equation}
As we shall see, this approach will be considered in the 3D case. Here we will stick on the previous approach, easy to deal with for the theoretical analysis.
\end{remark}
%
\subsection{Subspaces on curved polygons declared as \emph{slave}}\label{subsec:VkP2d_slave}

\noindent
Let $\PP'$ be a slave element, which is separated by the corresponding master element $\PP$ by the curved edge $\xi$. The space $V^k(\PP')$ is defined as in~\eqref{eq:VkP2d_master}-\eqref{eq:VkP2d_master-final}, but care is needed when defining the corresponding basis functions, since the behavior on $\xi$ is determined by the master polygon $\PP$. We proceed in a way similar to \cite{beirao2020}; in \cite{beirao2020} a ficticious set of generating points and a corresponding set of basis functions was introduced to define traces on the curved edge. Here the set of generating points is naturally provided by the dofs \eqref{eq:VkP2d_master_dofs} on the master element, and the generators are the corresponding basis functions. Precisely, let $N_{\PP}$ be the dimension of $V^k(\PP)$, and let $\{\varphi_i\},\,i=1,\cdots, N_{\PP}$ be the corresponding canonical basis functions defined as $\mbox{dof}_i(\varphi_j)=\delta_{ij}$, and with trace on $\xi$ computed as in \eqref{eq:lsq2d}-\eqref{eq:lsq2d_sol}. Let $N_{\PP'}$ be the number of dofs \eqref{eq:VkP2d_master_dofs} of $\PP'$, and let $\{{\widetilde \varphi}_i\},\,i=1,\cdots, N_{\PP'}-2$ be a subset of basis functions in $V^k(\PP')$ associated with the dofs not positioned on $\xi$, defined as
\begin{equation}\label{basis-slave-zero}
\mbox{ dof}_i({\widetilde\varphi}_j)=\delta_{ij}\quad \mbox{and } {{\widetilde\varphi}_i}= 0 \mbox{ on }\xi\quad i,j=1,\cdots,N_{\PP'}-2.
\end{equation}
The basis functions for the space $V^k(\PP')$ will be:
\begin{itemize}
\item the $N_{\PP}-2$ basis functions $\varphi_i$ in $V^k(\PP)$, extended to $\PP'$ as
\begin{equation}\label{eq:basis-slave-from-master}
\varphi_i= 0 \mbox{ on the straight edges of } {\PP'},\mbox{ and }\int_{\PP'} \varphi_i q= 0 \quad \forall q \in {\mathbb P}_{k-2}(\PP');
\end{equation}
\item the two basis functions associated with the endpoints of $\xi$, extended to $\PP'$ as:
\begin{align*}
&\varphi_i =0 \mbox{ on the vertices of }{\PP'} \not\in\xi,\\
&\varphi_i =0 \mbox{ on the }k-1 \mbox{ Gauss-Lobatto points of the straight edges of  }\PP',\\
&\int_{\PP'} \varphi_i q= 0 \quad \forall q \in {\mathbb P}_{k-2}(\PP');
\end{align*}
\item the $N_{\PP'}-2$ functions ${\widetilde\varphi}_i$ defined in \eqref{basis-slave-zero}, extended to zero in ${\PP}$.
\end{itemize}

\noindent
Also in this case, for every $k \geq 1$,
\begin{equation*}
    \mathbb P_k(\PP') \subseteq V^k(\PP').
\end{equation*}

\subsection{Subspaces on polygons with a curved edge in $\Gamma_D$}\label{subsec:VkP2d_diri}

\noindent
On elements $\PP$ with one curved edge $\xi\subset\Gamma_D$ the situation is simple. Since the trace is assigned, our local spaces will be:
\begin{multline}\label{eq:VkP2d_diri}
    {\widetilde V}^k_{g_D}(\PP) := \{\,v \in H^1(\PP) \cap\,C^0(\overline\PP)  :v_{|e} \in \mathbb P_k(e)\,\,\forall\text{ straight edge }e\subset\partial\PP,\\
    v_{|\xi} = g_D,  \,\Delta v \in\mathbb P_{k}(\PP)\,\},
\end{multline}
and then
\begin{equation}\label{eq:VkP2d_diri-final}
V^k_{g_D}(\PP)=\{ v \in {\widetilde V}^k_{g_D}(\PP) | \int_{\PP} v q_s \,dx = \int_{\PP}\Pi^\nabla_{k,\PP} v q_s \,dx\quad \forall q_s\in {\mathbb P}^{\rm hom}_s(\PP), \,s=k-1,k\}.
\end{equation}
An element of $V^k_{g_D}(\PP)$ is fully determined by the knowledge of $g_D$ and the set of degrees of freedom \eqref{eq:VkP2d_master_dofs}. For a general $g_D$, $\mathbb P_k(\PP) \not\subset V^k_{g_D}(\PP)$, but whenever $g_D$ is the trace of a polynomial $p_k\in\mathbb P_k$, then the condition $\mathbb P_k(\PP) \subset V^k_{g_D}(\PP)$ will hold.

\subsection{Subspaces on polygons with a curved edge in $\Gamma_N$}\label{subsec:VkP2d_Neumann}

On elements $\PP$ with one curved edge $\xi\subset\Gamma_N$, we can define local spaces analogous to~\eqref{eq:VkP2d_master-final}. In this case we need to compute integrals
\begin{equation}\label{int-Neu}
\int_{\xi} g_N v \, d \xi.
\end{equation}
This can be done since, thanks to \eqref{eq:lsq2d}-\eqref{eq:lsq2d_sol}, the trace of $v$ is a known polynomial on $\xi$. 

\begin{lemma}\label{lemma:2dcomputability-plain}
    For all types of elements defined in Sections~\ref{subsec:VkP2d_straight}, \ref{subsec:VkP2d_master}, \ref{subsec:VkP2d_slave}, and \ref{subsec:VkP2d_diri}, the degrees of freedom allow to compute exactly
    \begin{gather*}
    \Pi^\nabla_{k,\PP}\colon H^1(\PP)\to\mathbb P_k(\PP),\quad
    \Pi^0_{k-1,\PP}\nabla\colon H^1(\PP)\to[\mathbb P_{k-1}(\PP)]^2,\\
    \Pi^0_{k,\PP}\colon L^2(\PP)\to \mathbb P_{k}(\PP).
    \end{gather*}
\end{lemma}
\begin{proof}
The proof is standard, see, e.g. \cite{Acta-VEM}.
\end{proof}
\begin{remark}
    Lemma~\ref{lemma:2dcomputability-plain} assumes the use of quadrature formula on (possibly) curved elements~\cite{chin2021,antolin2022}.
\end{remark}

%

\section{The Discrete Problem in Two Dimensions}\label{sec:discrete2d}

\noindent
Let $\{\mathcal T_h\}_h$ be a sequence of conforming partitions of $\Omega\subset\mathbb R^2$ by (possibly curved) polygons, and let $h = \max_{\PP\in\mathcal T_h}h_{\PP}$. 
We assume that the decomposition respects the discontinuities of $\kappa$. 
For every given function $\eta\in H^{1/2}(\Gamma_D)$, 
we define the global space
\begin{multline}\label{eq:Vkh}
    V^k_{\eta,h} := \{\,v \in H^1({\Omega})\,|\,v_{|\PP} \in V^k(\PP)\quad\forall\,\PP\in\mathcal T_h\text{ without edges in }\Gamma_D\\
    \text{and }v_{|\PP} \in V^k_{\eta}(\PP)\quad\forall\,\PP\in\mathcal T_h\text{ with an edge in }\Gamma_D\}.
\end{multline}
The degrees of freedom for $V^k_{\eta,h}$ will be:
\begin{equation}\label{eq:Vkh_dofs}
    \begin{aligned}
        &\bullet\,\,\text{the values at the vertices of $\Omega\cup\Gamma_N$},\\
        &\bullet\,\,\text{(for $k \geq 2$) the values of $v$ at the $k-1$ Gauss-Lobatto points}\\
        &\qquad\text{of each straight edge internal to $\Omega$ or in $\Gamma_N$},\\
        &\bullet\,\,\text{(for $k \geq 2$) the moments of order $\leq k -2$ internal to each element}.
    \end{aligned}    
\end{equation}
%
%

\noindent
The discrete problem for~\eqref{eq:model-weak} is
\begin{equation}\label{eq:discrete2d}
\begin{cases}
\text{find } u_h \in V^k_{g_D,h} \text{ such that}\\
a_h(u_h, v) = (f_h, v) + \langle g_N, v\rangle\quad\forall\,v \in V^k_{0,h},
\end{cases}
\end{equation}
where
\begin{align}
    &a_h(w, v) = \sum_{\PP\in\mathcal T_h}a_h^{\PP}(w, v),\label{eq:ah}\\
    &a_h^{\PP}(w, v) = \kappa_{|\PP}\int_{\PP}(\Pi^0_{k-1,\PP}\nabla w)\cdot(\Pi^0_{k-1,\PP}\nabla v)\,dx + \kappa_{|\PP}S^{\PP}(w-\Pi^\nabla_{k,\PP}w, v-\Pi^\nabla_{k,\PP}v),\label{eq:ahP}\\
    &{f_h}_{|\PP} = \Pi^0_{k-1} f \quad \forall {\PP\in\mathcal T_h}.\label{eq:fh}
\end{align} 
In    \eqref{eq:ahP} $S^{\PP}(v,w)$ is, as usual, a symmetric positive semi-definite bilinear form such that  there exists a positive constant $\alpha_*$ independent of $h_{\PP}$ such that
\begin{equation}\label{eq:Pcoercivity}
\alpha_* a^{\PP}(v, v) \leq a^{\PP}_h(v, v) 
\end{equation} 
%
where $a^{\PP}(w,v) = \int_{\PP}\kappa \nabla w\cdot\nabla v\,dx$. 

Here we shall use the most common expression
  \begin{equation}\label{stab-dofidofi}
    S^{\PP}(w, v) = \sum_{i = 1}^{N_{\PP}}\text{dof}_i(w)\text{dof}_i(v)\quad\text{($N_{\PP}$ being the number of dofs in $\PP$).}
    \end{equation}
    For other choices of $S^{\PP}$ see  \cite{Wriggers-1} and \cite{beirao2020}.
 We note that, if one of the  two entries of $a_h^{\PP}(\cdot, \cdot)$ is a $q_k\in \mathbb P_k$, then
 $\Pi^\nabla_{k,\PP} q_k \equiv q_k$,\, $\Pi^0_{k-1,\PP}\nabla q_k \equiv \nabla q_k$. Hence,by the definition of projection:
 \begin{equation}\label{consistency}
 a_h^{\PP}(w, q_k)=\kappa_{|\PP}\int_{\PP}(\Pi^0_{k-1,\PP}\nabla w)\cdot \nabla q_k\,dx=a^{\PP}(w, q_k).
 \end{equation}



\begin{remark}
    Whenever the exact solution $u$ is in $\mathbb P_k(\Omega)$, it holds $u_h \equiv u$, so that the patch test is satisfied. 
    \end{remark}

\medskip\noindent

\begin{remark}\label{remark-Pcoercivity} 
 For coercivity~\eqref{eq:Pcoercivity} to hold, stabilization must be applied in each element. Thus, the degrees of freedom common to two or more elements would be stabilized more than once. In general, this does not jeopardize the performance of the method, but when a curved edge $\xi$ separates two elements (say, $\PP$ and $\PP'$) where $\kappa$ assumes very different values, stabilizing only once is the only way to preserve accuracy. We refer to  \cite{beirao2020} for a detailed discussion on the subject. Fortunately, in \cite{beirao2020} it is also shown that inequality \eqref{eq:Pcoercivity} does not need to hold in each element necessarily. It is enough that
\begin{equation}\label{eq:global_coercivity}
\alpha_* a(v, v) \leq a_h(v, v)\quad\forall\,v\in V^k_{0,h}.
\end{equation}
Hence, for every curved edge $\xi$ separating a master element $\PP$ and a slave element $\PP'$, we will stabilize the dofs defining the value on $\xi$ only when dealing with the master element $\PP$.
With our assumptions on the mesh, coercivity \eqref{eq:global_coercivity} can be proved by adapting the arguments in \cite{beirao2017} and  \cite{brenner2017}.
\end{remark}

\subsection{Implementing non-homogeneous Dirichlet boundary conditions}

The solution of~\eqref{eq:discrete2d} can be decomposed as
\begin{equation}\label{eq:lifting}
u_h = u_h^0 + \overline g_D,
\end{equation}
where $u_h^0\in V^k_{0,h} \subset H_{0,\Gamma_D}^1(\Omega)$ and $\overline g_D \in V^k_{g_D,h}$ is a lifting of the boundary datum. Typically, $\overline g_D$ has all the dofs~\eqref{eq:Vkh_dofs} equal to zero.  Therefore, the bilinear form $a_h$ splits into
\[
a_h(u_h^0, v) + a_h(\overline g_D, v),\quad v\in V^k_{0,h}.
\]
The second term is zero on all the elements except on those having an edge, or a vertex, on $\Gamma_D$ and possibly on neighbouring slave elements having a nonzero ${\overline g}_D$ on the curved edge. On such elements, from~\eqref{eq:ahP}, we get
\[
a_h^{\PP}(\overline g_D, v) = \kappa_{|\PP}\int_{\PP}(\Pi^0_{k-1,\PP}\nabla \overline g_D)\cdot(\Pi^0_{k-1,\PP}\nabla v)\,dx + \kappa_{|\PP}S^{\PP}(\overline g_D-\Pi^\nabla_{k,\PP}\overline g_D, v-\Pi^\nabla_{k,\PP}v),
\]
where we recall that $S^\PP(\cdot,\cdot)$ is not computed for slave elements. The terms $\Pi^\nabla_{k,\PP}\overline g_D$ and $\Pi^0_{k-1,\PP}\nabla\overline g_D$ are obtained from~\eqref{eq:PiNablaP} and \eqref{eq:PiGradP}, respectively: 
\begin{align*}
    \int_{\PP}\nabla\Pi^\nabla_{k,\PP}\overline g_D\cdot\nabla q\,dx &= \int_{\partial\PP}\overline g_D(\mathbf n\cdot\nabla q)\,ds - \int_{\PP}\overline g_D \Delta q\, dx\\
    &= {\sum_{e\subset\partial\PP\cap\Gamma_D}}\int_e g_D(\mathbf n\cdot\nabla q)\,ds + \sum_{e\not\subset\partial\PP\cap\Gamma_D\colon \overline g_{D|e}\neq 0}\int_e \overline g_D(\mathbf n\cdot\mathbf q)\,ds - 0,
\end{align*}
\begin{align*}
    \int_{\PP}\Pi^0_{k-1,\PP}\nabla \overline g_D\cdot\mathbf q\,dx &= \int_{\partial P}\overline g_D(\mathbf n\cdot\mathbf q)\,ds - \int_{\PP}\overline g_D \text{ div}\,\mathbf q\,dx\\
    &= {\sum_{e\subset\partial\PP\cap\Gamma_D}}\int_e g_D(\mathbf n\cdot\mathbf q)\,ds + \sum_{e\not\subset\partial\PP\cap\Gamma_D\colon \overline g_{D|e}\neq 0}\int_e \overline g_D(\mathbf n\cdot\nabla q)\,ds - 0.
\end{align*}
The discrete problem can now be rewritten as
\begin{equation}\label{eq:discrete2d-symmetric}
\begin{cases}
\text{find } u_h^0 \in V^k_{0,h} \text{ such that}\\
a_h(u_h^0, v) = (f_h, v) + \langle g_N, v\rangle - a_h(\overline g_D, v)\quad\forall\,v \in V^k_{0,h}.
\end{cases}
\end{equation}

\subsection{An abstract error estimate}

\noindent
Let $\|v\|_{0,E}$ denote the $L^2(E)$ norm, $|v|_{1,E}$ and $\|v\|_{1,E}$ the $H^1(E)$ semi-norm and norm, respectively. When we refer to the whole domain $\Omega$, we will simply write $\|v\|_0, |v|_1$, and $\|v\|_1$.

Let $|v|_{1,h}$ and $\|v\|_{1,h}$ be the broken $H^1$ semi-norm and norm  given by
\begin{equation}\label{eq:brokenH1}
    |v|_{1,h} = \left(\sum_{\PP\in\mathcal T_h}|v|^2_{1,\PP}\right)^{1/2}, \quad \|v\|_{1,h} = \left(\sum_{\PP\in\mathcal T_h}\|v\|^2_{1.\PP}\right)^{1/2}.
\end{equation}

Let $\|\cdot\|_h = \sqrt{a_h(\cdot,\cdot)}$ be the mesh-dependent energy norm, and observe that the symmetry of $a_h(\cdot,\cdot)$ implies continuity:
\begin{equation}\label{eq:ah_continuity}
    a_h(w, v) \leq (a_h(w, w))^{1/2} (a_h(v, v))^{1/2}= \|w\|_h \|v\|_h\quad w, v\in V_{0,h}^k(\Omega).
\end{equation}
Moreover, the stability condition \eqref{eq:global_coercivity} can be rewritten as:
\begin{equation}\label{eq:global_coercivity-h}
\alpha_* a(v, v) \leq a_h(v, v) := \|v\|^2_h \quad\forall\,v\in V^k_{0,h},
\end{equation}
implying
\begin{equation}\label{ellipticity-H1}
C\, \| v \|_1^2 \leq a(v, v) \leq \|v\|^2_h \quad\forall\,v\in V^k_{0,h} .
\end{equation}

\begin{theorem}\label{Theo-abstract}
Problem \eqref{eq:discrete2d} has a unique solution. Moreover, if $u$ is the solution of \eqref{eq:model-weak},  $\uI$ any function in $V^k_{g_D,h}$, and $\up$  any function elementwise in $\mathbb P_k$, the following estimate holds:
    \begin{equation}\label{eq:abstract_error}
        \|u - u_h\|_{1} \leq C \Big( \|u-\uI\|_1 +  \|\uI-\up\|_h  + \|\up-u\|_{1,h}+ \|f-f_h\|_0 )
    \end{equation}
    being $C$ a constant depending on $\alpha_*$ and the data $\kappa_{\min} = \min_{\PP\in\mathcal T_h}\kappa_{|\PP}, \kappa_{\max} = \max_{\PP\in\mathcal T_h}\kappa_{|\PP}$.
 \end{theorem}
\begin{proof}
Uniqueness is a consequence of \eqref{eq:global_coercivity-h}. Setting $\delta=u_h-u_I\in V_{0,h}^k$, the proof goes in the lines of \cite{beirao2020}.
 By using \eqref{eq:discrete2d}, $\pm \up$, \eqref{consistency}, $\pm u$, \eqref{eq:model-weak}, \eqref{eq:ah_continuity}, and finally \eqref{ellipticity-H1} we obtain
    \begin{equation}\label{eq:abstract_error_step1}
    \begin{split}
        \|\delta\|^2_h & = a_h(u_h-\uI, \delta) =(f_h, \delta) + \langle g_N, \delta\rangle -a_h(\uI, \delta)\\
        &=(f_h, \delta) + \langle g_N, \delta\rangle -a_h(\uI-\up, \delta)-\sum_{\PP} a^{\PP}(\up,\delta)\\
        &=(f_h, \delta) + \langle g_N, \delta\rangle -a_h(\uI-\up, \delta)-\sum_{\PP} a^{\PP}(\up-u,\delta)-a(u,\delta)\\
        &=(f_h-f, \delta) -a_h(\uI-\up, \delta) -\sum_{\PP} a^{\PP}(\up-u,\delta)\\
        &\leq  C \Big(\|f-f_h\|_0 \|\delta\|_1 + \|\uI-\up\|_h \|\delta\|_h + \|\up-u\|_{1,h}\|\delta\|_1\Big)\\
         &\leq  C \Big(\|f-f_h\|_0  + \|\uI-\up\|_h  + \|\up-u\|_{1,h}\Big)\|\delta\|_h .   
         \end{split}
    \end{equation}
    Then, the result follows from \eqref{ellipticity-H1} and the triangle inequality.
\end{proof}
From Theorem \ref{Theo-abstract} we see that, in order to have the usual estimates in terms of powers of $h$ and regularity of the solution, we need to estimate the four terms in \eqref{eq:abstract_error}. The last term is easy, thanks to \eqref{eq:fh} and standard polynomial approximation results:
\begin{equation}\label{error-f}
\|f - f_h \|_0 \le C\, h^k \Big(\sum_{\PP} |f|^2_{k,\PP}\Big)^{1/2}.
\end{equation}
For the other terms we need some work.

\subsection{Definition of $\uI$ and $\up$}
Let $\PP$ be an element of $\mathcal{T}_h$, and let $p^*_k\in \P_k(\PP)$ be the polynomial associated with a function $u\in H^s(\PP),\,s\ge 2$, defined as in \eqref{eq:lsq2d}. Following \cite{beirao2020} we define
\begin{equation}\label{def:uI} 
\begin{aligned}
&\uI\in V^k(\PP) \mbox{ such that}\\ 
&\uI=u \mbox{ at the vertices of }\PP, \\
&\uI=u \mbox{ and at the Gauss-Lobatto points of the straight edges}, \\
&\uI=p^*_k \mbox{ on the curved edge $\xi$, and }\\
&\Delta \uI=\Pi^0_{k-2,\PP} \Delta u. 
\end{aligned}
\end{equation}
The following estimate can be proved by adapting the arguments of  \cite{beirao2020}, where it was proved that, under our assumptions on the mesh, there exists a positive constant $C$ such that, for all $\PP$ 
\begin{equation}\label{error-uI}
|u-\uI|_{1,\PP} \leq C h^{s-1}_{\PP} |u|_{s,\PP} \quad s\ge 2.
\end{equation}
Let $\up\in \P_k$ be the $L^2$-projection of $u$, for which we have
\begin{equation}\label{error-up}
\|u - \up \|_{1,\PP}:=\|u - \Pi^0_{k} u\|_{1,\PP} \le C h^{\min\{k,s-1\}}_{\PP}|u|_{s,\PP}.
\end{equation}
It remains to estimate the term $\|\uI-\up\|_h$. We have:
\begin{equation}
\begin{aligned}
\|\uI-\up\|^2_h =&a_h(\uI-\up, \uI-\up):= \sum_{\PP}\kappa_{|\PP}\int_{\PP}|\Pi^0_{k-1,\PP}\nabla (\uI-\up)|^2\,dx \\
&+ \sum_{\PP}\kappa_{|\PP}S^{\PP}((I-\Pi^\nabla_{k,\PP})(\uI-\up), (I-\Pi^\nabla_{k,\PP})(\uI-\up)).
\end{aligned}
\end{equation}
The first term is easily estimated from \eqref{error-uI}-\eqref{error-up}. The second term is also standard, upon noticing that our choice \eqref{stab-dofidofi} for the stabilization can be treated exactly as for straight polygons, see \cite{beirao2020}. Indeed, $S^{\PP}(v,v)$ splits into a part $S^{int}$ involving internal moments and a part $S^{\delta}$ involving boundary values. $S^{int}$ is easily bounded by observing that each moment \eqref{eq:VkP2d_straight_dofs} is bounded by $\|v\|_{0,\PP}h_{\PP}^{-1}$. Thus,
\begin{equation}\label{stima-SE-int}
S^{int}(v,v)\lesssim \|v\|^2_{0,\PP}h_{\PP}^{-2}.
\end{equation}
Next, denoting by $N$ the number of boundary nodes $\nu_i$, we have
\begin{equation}\label{stima-SE-boundary}
\begin{aligned}
S^{\delta}(v,v) &= \sum_{i=1}^{N} v(\nu_i)^2 \le N \|v\|^2_{\infty,\partial\PP}\lesssim h_{\PP}^{-1}\|v\|^2_{0,\partial\PP}+ |v|^2_{1/2,\partial\PP}\\
&\lesssim h_{\PP}^{-2}\|v\|^2_{0,\PP}+ |v|^2_{1,\PP}.
\end{aligned}
\end{equation}
Under assumption \ref{mesh2D-star} on the mesh  we have
\begin{equation}\label{intermediate}
h_{\PP}^{-2}\|v\|^2_{0,\PP}\lesssim |v|^2_{1,\PP}.
\end{equation}
Using \eqref{intermediate} in \eqref{stima-SE-int}-\eqref{stima-SE-boundary} we deduce
\begin{equation}\label{stima-final-SE-1}
S^{\PP}(v,v) \lesssim |v|^2_{1,\PP}.
\end{equation}
Hence, from \eqref{stima-final-SE-1} we obtain
\begin{equation}\label{stima-final-SE}
\begin{aligned}
&S^{\PP}((I-\Pi^\nabla_{k,\PP})(\uI-\up), (I-\Pi^\nabla_{k,\PP})(\uI-\up))\\
&=S^{\PP}(\uI-\Pi^\nabla_{k,\PP}\uI,\uI-\Pi^\nabla_{k,\PP}\uI)
\lesssim |\uI-\Pi^\nabla_{k,\PP}\uI|^2_{1,\PP}\\
&\lesssim |\uI- u|^2_{1,\PP}+ |u-\Pi^\nabla_{k,\PP}u|^2_{1,\PP} +|\Pi^\nabla_{k,\PP}(u-\uI)|^2_{1,\PP},
\end{aligned}
\end{equation}
so that we finally obtain
\begin{equation}\label{error-ui-up}
\|\uI-\up\|_h \le C\, h^{s-1} |u|_{s} \quad s\ge 2.
\end{equation}
\begin{corollary}\label{final-estimate}
Under our assumptions on the mesh, the following estimate holds:
\begin{equation}
\|u -u_h\|_{1} \le C\, h^k |u|_{k+1}.
\end{equation}
\end{corollary}
\begin{proof}
By collecting \eqref{error-uI}, \eqref{error-up}, \eqref{error-ui-up}, and \eqref{error-f} in \eqref{eq:abstract_error} the result follows.
\end{proof}

\section{Virtual Element Method in Three Dimensions}\label{sec:VEM3d}

\renewcommand{\PP}{\mathcal P}

Let $\{\mathcal T_h\}_h$ be a sequence of conforming partitions of $\Omega\subset\mathbb R^3$ by (possibly curved) polyhedra $\PP$, each with diameter $h_\PP$. The decomposition is assumed to respect the discontinuities of $\kappa$. We assume that $\PP$ has at most one face that will be treated as ``curved'' (geometrically speaking, it could be either curved or flat or almost flat). All the other faces are assumed to be flat. This condition is mandatory for our approach. Note that $\PP$ is still allowed to have curved edges not attached to the curved face (see the numerical tests in Section~\ref{test3d-bubble}).

We make the following assumption on the decomposition (see, e.g., \cite{max3}): there exists a real number $\rho_0 > 0$ such that each polyhedron $\PP$ satisfies the following conditions:
\begin{enumerate}[label=(B\arabic*)]
\item\label{mesh3D-C1} the boundary of $\PP$ is piecewise $C^1$;
\item\label{mesh3D-star} $\PP$ is star-shaped with respect to a ball of radius $\geq \rho_0 h_{\PP}$;
\item\label{mesh3D-face-star} every face $\ff$ of is star-shaped with respect to a disk of radius $\geq \rho_0 h_{\ff}\geq \rho^2_0 h_{\PP}$;
\item\label{mesh3D-rho} every edge $e$ of $\ff$ has length  $h_e \geq \rho^2_0  h_\PP$.
   \end{enumerate}
We can further relax the above assumptions~\cite{beirao2017,brenner2017,brenner2018,beirao_curvo_2019}.

Now, let us introduce the local spaces. First,  let us recall the definition of the local spaces on  polyhedra with flat faces, see \cite{max3} or \cite{BBDMR-Cina}. Let:
\begin{equation}\label{def:localspaces}
\widetilde V^k(\PP):=\{ v\in H^1(\PP)\cap C^0(\overline\PP) :\, 
v_{|\ff}\in V^k(\ff) \, \forall \mbox{ face } {\ff}, \,\Delta v \in \P_{k}(\PP) \},
\end{equation}
with $V^k(\ff)$ defined in \eqref{eq:VkP2d_straight-final}.
In $\widetilde V^k(\PP)$  the following operators are well defined:
\begin{equation}\label{eq:VkP3d_straight_dofs}
    \begin{aligned}
    &\bullet\,\,\text{the values of $v$ at each vertex of $\mathcal P$,}\\
    &\bullet\,\,\text{(for $k\geq 2$) the values of $v$ at the $k-1$ Gauss-Lobatto points}\\
    &\qquad \text{of each edge $e\subset\partial\mathcal P$,}\\
    &\bullet\,\,\text{ (for $k \geq 2$) the moments of order up to $k-2$ on each face {\ff},}\\
      &\bullet\, \, \text{ (for $k \geq 2$) the moments of order up to $k-2$ in {$\PP$}. We use a scaled}\\
      &\qquad \text{polynomial basis $\{q_\beta\}_\beta$ of $\mathbb P_{k-2}(\PP)$, with}\\
      &\qquad q_\beta(x,y,z) = \left(\frac{x-x_{\PP}}{h_{\PP}}\right)^{\beta_1}\left(\frac{y-y_{\PP}}{h_{\PP}}\right)^{\beta_2}\left(\frac{z-z_{\PP}}{h_{\PP}}\right)^{\beta_3}\\
      &\qquad\qquad\qquad\qquad\qquad \forall\,\beta=(\beta_1, \beta_2, \beta_3)\in\mathbb N^3,\quad\beta_1+\beta_2+\beta_3 \leq k-2.
    \end{aligned}
\end{equation}
In particular, they allow to compute $\Pi^{\nabla}_{k,\PP} v, \forall v \in \widetilde V^k(\PP)$.

As we did in 2d, let us first examine the case of a tetrahedron $\PP$, with three flat faces, and one curved face $\sigma$. Then, each flat face has one curved edge $\xi$, and we proceed as we did in\eqref{eq:lsq2d_tri}. This will allow to define a continuous trace of $v$ on the whole boundary of $\sigma$. Next, the natural extension to define $v$ on $\sigma$ would be to use the values at the four vertices and the degrees of freedom not positioned on $\sigma$ to individuate a polynomial in $\P_k(\PP)$ and use its trace on $\sigma$. A simple count of the degrees of freedom shows that they amount to $4+ 3(k-1)+ 3k(k-1)/2 + (k-1)k(k+1)/6$, more than the dimension of $\P_k(\PP)$. Using the internal moments up to the order $k-3$ instead of $k-2$ fixes the problem. Indeed:
\begin{equation}\label{dim-Pk-3d}
4+ 3(k-1)+ \frac{3k(k-1)}{2} + \frac{k(k-1)(k-2)}{6}\equiv dim\,\P_k(\PP).
\end{equation}
Hence, we can compute a polynomial $p^*_k\in \P_k$ through the conditions
\begin{equation}\label{eq:VkP3d_teth_dofs}
    \begin{aligned}
      &\bullet\,\, p^*_k =v \mbox{ at each vertex of } \PP,\\
    &\bullet\,\,\text{(for $k\geq 2$)\quad $p_k^* = v$ at the $k-1$ Gauss-Lobatto points}\\
    &\qquad \forall\, \text{{\bf straight} edge } e,\\
    &\bullet\,\,\text{ (for $k \geq 2$) } \int_{\ff} p^*_k q_{k-2} =\int_{\ff} v q_{k-2} \quad \forall q_{k-2}\in \P_{k-2}({\ff}),\,\forall \text{ {\bf flat} face } {\ff},\\
    &\bullet\, \, \text{ (for $k \geq 3$) } \int_{\PP} p^*_k q_{k-3} =\int_{\PP} v q_{k-3} \quad \forall q_{k-3}\in \P_{k-3}({\PP}).  
    \end{aligned}
\end{equation}
At this point the natural extension of what we did in 2d would be to set $v_{|\sigma}= p^*_k $. This cannot be done in 3d since, in general, the trace of  $p^*_k $ on $\partial \sigma$ does not coincide with the trace coming from the flat faces. Therefore we shall use the procedure described in Remark \ref{alternative-procedure}, copying the moments of $p^*_k $ up to the order needed. In particular, we need to compute $\Pi^\nabla_{k,\PP} v$. Looking at the definition \eqref{eq:PiNablaP} we see that, upon integration by parts, we need to compute $\int_{\sigma} v \nabla q_{k}\cdot {\bf n} \,d\sigma$ and $\int_{\sigma} v \,d\sigma$. Then we set
\begin{equation}\label{mimetic-3D}
\int_{\sigma} v \nabla q_{k}\cdot {\bf n} \,d\sigma := \int_{\sigma} p_k^* \nabla q_k \cdot {\bf n}\,d\sigma \quad \forall q_k\in\P_k(\sigma),\quad \int_{\sigma} v \,d\sigma:= \int_{\sigma} p_k^* \,d\sigma.
\end{equation}
On the other hand, we also need to compute $\Pi^0_{k-1} \nabla v$, for which we need to compute $\int_{\sigma} v {\bf q}_{k}\cdot {\bf n} \,d\sigma$. Hence we set
\begin{equation}\label{Pi0-3d}
\int_{\sigma} v {\bf q}_{k-1}\cdot {\bf n} \,d\sigma := \int_{\sigma} p^*_k {\bf q}_{k-1}\cdot {\bf n} \,d\sigma.
\end{equation}

To play safe, a function $v$ will be individuated by conditions \eqref{eq:VkP3d_straight_dofs} plus the moments of 
$p_k^*$ up to order $k$  on $\sigma$.

\noindent
The analogues of the spaces \eqref{eq:VkP2d_master}-\eqref{eq:VkP2d_master-final} will be
\begin{multline}\label{eq:VkP3d_master-mimetic}
    {\widetilde V}^k(\PP) := \{\,v \in H^1(\PP)\cap C^0(\overline\PP)\, : \,v_{|\ff} \in V^k(\ff)\,\,\forall\text{ flat face }{\ff}\subset\partial\PP,\\
    \int_{\sigma} v q_{k} \,d\sigma = \int_{\sigma}  p_k^* q_{k} \,d\sigma \quad \forall q_{k} \in \mathbb P_{k}(\PP),  \,\Delta v \in\mathbb P_{k}(\PP)\,\},
\end{multline}
and then
\begin{equation}\label{eq:VkP3d_master-final-mimetic}
V^k(\PP)=\{ v \in {\widetilde V}^k(\PP)\, :\, \int_{\PP} v q_s \,dx = \int_{\PP}\Pi^\nabla_{k,\PP} v q_s \,dx\quad \forall q_s\in {\mathbb P}^{\rm hom}_s(\PP), \,s=k-1,k\}.
\end{equation}
As before, we still have, for every $k \geq 1$,
\begin{equation*}
    \mathbb P_k(\PP) \subseteq V^k(\PP).
\end{equation*}
For a generic polyhedron with one curved face $\sigma$ the number of conditions \eqref{eq:VkP3d_teth_dofs} is bigger than the dimension of 
$\P_k(\PP)$, so that we use a least square approach like we did in 2d (see \eqref{eq:lsq2d}). Below, we summarize the whole procedure. First, for $v \in V^k(\PP)$, introduce the following well defined operators:
\begin{equation}\label{eq:VkP3d_master_dofs}
    \begin{aligned}
    &\bullet\,\,\text{the values of $v$ at each vertex of $\PP$,}\\
    &\bullet\,\,\text{(for $k\geq 2$) the values of $v$ at the $k-1$ Gauss-Lobatto points}\\
    &\qquad \text{of each straight edge $e\subset\partial\PP$,}\\
    &\bullet\,\,\text{ (for $k \geq 2$) the moments of order up to $k-2$ on each flat face {\ff},}\\
    &\bullet\, \, \text{ (for $k \geq 2$) the moments of order up to $k-2$ in {$\PP$}}.
    \end{aligned}
\end{equation}
Then, we proceed as follows:
\begin{itemize}
\item On each flat face {\ff} with one curved edge $\xi$ we compute a polynomial $p^{\ff}_k\in \P_k(\ff)$ as in \eqref{eq:lsq2d} and set $v_{|\xi}=p^{\ff}_k$. This will ensure continuity of $v$ on $\partial \sigma$ and allow to compute $\Pi^{\nabla,\ff}_k v$ on each face.
\item We compute a polynomial $p^*_k \in \P_k(\PP)$ as the {\bf unconstrained} least square solution of \eqref{eq:VkP3d_teth_dofs}. (Note that the functions having internal moments of order $k-2$  will provide $p^*_k \equiv 0$ by construction.)
\item We apply \eqref{mimetic-3D} and \eqref{Pi0-3d} in order to compute $\Pi^{\nabla,\PP}_k v$ and $\Pi^0_{k-1} \nabla v$, so that the discrete bilinear form is completely computable.
\end{itemize}
The above procedure applies to elements with a curved face $\sigma$ at the interior, and refers to a "master" element. The neighbouring element sharing the curved face will be treated  as a "slave", like we did in the 2d case. More precisely, let $\PP, \PP'$ denote a master-slave pair of elements. Also, let $N_\PP$ be the dimension of $V^k(\PP)$, and let $\{\varphi_i\}, i = 1, \dots, N_\PP$ be the corresponding canonical basis functions defined as $\text{dof}_i(\varphi_j) = \delta_{ij}$ and with moments on $\sigma$ computed as just mentioned above. Let $N_\sigma$ be the number of dofs positioned on $\sigma$, that is, the number of vertices of $\sigma$. Let $N_{\PP'}$ be the number of dofs~\eqref{eq:VkP3d_master_dofs} of $\PP'$, and let $\{\widetilde\varphi_i\}, i = 1, \dots, N_{\PP'}-N_\sigma$ be a subset of basis functions in $V^k(\PP')$ associated with the dofs not positioned on $\sigma$, defined as
\begin{equation}\label{eq:basis-slave-zero-3d}
\text{dof}_i({\widetilde\varphi}_j)=\delta_{ij}\quad \text{ and } {{\widetilde\varphi}_i}= 0 \text{ on }\sigma\quad i,j=1,\cdots,N_{\PP'}-N_\sigma.
\end{equation}
The basis functions for the space $V^k(\PP')$ will be:
\begin{itemize}
\item the $N_{\PP}-N_\sigma$ basis functions $\varphi_i$ in $V^k(\PP)$, extended to $\PP'$ as:
  \begin{align*}
      &\varphi_i= 0 \mbox{ on the straight edges of } {\PP'},\\
      &\int_{\ff} \varphi_i\, q_{k-2} = 0 \quad \forall q_{k-2}\in \P_{k-2}({\ff}),\,\forall \text{ {\bf flat} face } {\ff\in\partial\PP'},\\
      &\int_{\PP'}\varphi_i\, q_{k-2} = 0 \quad \forall q_{k-2}\in \P_{k-2}(\PP');
  \end{align*}
\item the $N_\sigma$ basis functions in $V^k(\PP)$ associated with the vertices of $\sigma$, extended to $\PP'$ as:
  \begin{align*}
    &\varphi_i = 0 \text{ on the vertices of }\PP'\not\in\sigma,\\
    &\varphi_i = 0 \text{ on the }k-1 \text{ Gauss-Lobatto points of the straight edges of  }\PP',\\    
    &\int_{\ff} \varphi_i\, q_{k-2} = 0 \quad \forall q_{k-2}\in \P_{k-2}({\ff}),\,\forall \text{ {\bf flat} face } {\ff\in\partial\PP'},\\
    &\int_{\PP'}\varphi_i\, q_{k-2} = 0 \quad \forall q_{k-2}\in \P_{k-2}(\PP');    
  \end{align*}
\item the $N_{\PP'}-N_\sigma$ functions ${\widetilde\varphi}_i$ defined in \eqref{eq:basis-slave-zero-3d}, extended to zero in ${\PP}$.
\end{itemize}
Again, for every $k \geq 1$,
\begin{equation*}
    \mathbb P_k(\PP') \subseteq V^k(\PP').
\end{equation*}


\subsection{Elements with a curved face on $\Gamma_D$}
Local spaces are analogous to~\eqref{eq:VkP3d_master-mimetic} and \eqref{eq:VkP3d_master-final-mimetic}, except for the trace on $\sigma$, which is assigned:
\begin{equation*}
  \begin{split}
  {\widetilde V}^k_{g_D}(\PP) := \{\,&v \in H^1(\PP)\cap C^0(\overline\PP)\, : \,v_{|\ff} \in V^k(\ff)\,\,\forall\text{ flat face }{\ff}\subset\partial\PP\setminus\Gamma_D,\\
  &v_{|\ff} \in V^k_{g_D}(\ff)\,\,\forall\text{ flat face }{\ff}\subset\partial\PP\text{ with a curved edge on }\Gamma_D,\\
  &v_{|\sigma} = g_D,  \,\Delta v \in\mathbb P_{k}(\PP)\,\},
  \end{split}
\end{equation*}
\begin{equation*}
V^k_{g_D}(\PP)=\{ v \in {\widetilde V}^k_{g_D}(\PP) : \int_{\PP} v q_s \,dx = \int_{\PP}\Pi^\nabla_{k,\PP} v q_s \,dx\quad \forall q_s\in {\mathbb P}^{\rm hom}_s(\PP), \,s=k-1,k\}.
\end{equation*}

\subsection{Elements with a curved face on $\Gamma_N$}
Here too we proceed differently from the 2d case, since we do not have traces on the curved faces, only moments. We define local spaces analogous to~\eqref{eq:VkP3d_master-final-mimetic}, then we approximate the Neumann datum with a polynomial and use the moments we have. More precisely:
\begin{equation}
\int_{\sigma} g_N\,v \,d\sigma \simeq \int_{\sigma} \Pi^0_{k} g_N\,v \,d\sigma = \int_{\sigma} \Pi^0_{k} g_N\,p^*_k \,d\sigma .
\end{equation}
In this way we introduce an approximation without affecting the order of convergence.

\subsection{The Discrete Problem in Three Dimensions}

For every given function $\eta\in H^{1/2}(\Gamma_D)$, we define the global space
\begin{multline*}
    V^k_{\eta,h} := \{\,v \in H^1({\Omega}) : v_{|\PP} \in V^k(\PP)\quad\forall\,\PP\in\mathcal T_h\text{ without faces on }\Gamma_D\\
    \text{and }v_{|\PP} \in V^k_{\eta}(\PP)\quad\forall\,\PP\in\mathcal T_h\text{ with a face on }\Gamma_D\}.
\end{multline*}
The degrees of freedom for $V^k_{\eta,h}$ will be:
\begin{equation*}
    \begin{aligned}
      &\bullet\,\,\text{the values at the vertices of $\Omega\cup\Gamma_N$},\\
      &\bullet\,\,\text{(for $k \geq 2$) the values of $v$ at the $k-1$ Gauss-Lobatto points}\\
      &\qquad\text{of each straight edge internal to $\Omega$ or in $\Gamma_N$},\\
      &\bullet\,\,\text{(for $k \geq 2$) the moments of order $\leq k -2$ internal to each flat face},\\
      &\bullet\,\,\text{(for $k \geq 2$) the moments of order $\leq k -2$ internal to each polyhedron}.
    \end{aligned}    
\end{equation*}
The discrete problem for~\eqref{eq:model-weak} is defined by~\eqref{eq:discrete2d}, where
\begin{align*}
    &a_h(w, v) = \sum_{\PP\in\mathcal T_h}a_h^{\PP}(w, v),\\
    &a_h^{\PP}(w, v) = \kappa_{|\PP}\int_{\PP}(\Pi^0_{k-1,\PP}\nabla w)\cdot(\Pi^0_{k-1,\PP}\nabla v)\,dx + \kappa_{|\PP}h_{\PP}S^{\PP}(w-\Pi^\nabla_{k,\PP}w, v-\Pi^\nabla_{k,\PP}v),\\
    &{f_h}_{|\PP} = \Pi^0_{k-1} f \quad \forall {\PP\in\mathcal T_h},
\end{align*}
and $S^\PP(\cdot,\cdot)$ is the three dimensional analogue of~\eqref{stab-dofidofi}.

\begin{remark}
If the exact solution $u$ belongs to $\mathbb P_k(\Omega)$, we get $u_h\equiv u$, so that the patch test is satisfied.
\end{remark}

\begin{remark}
The considerations in Remark~\ref{remark-Pcoercivity} regarding coercivity of the discrete problem and stabilization apply also to the three dimensional setting. Hence, for every curved face $\sigma$ separating a master element $\PP$ and a slave element $\PP'$, we will stabilize the dofs defining the value on $\sigma$ only when dealing with the master element $\PP$.
\end{remark}

\begin{remark}
We believe that coercivity of the discrete problem and error estimates could be proved by adapting the arguments in~\cite{brenner2018}. However, note that on a curved face belonging to a master elements or to $\Gamma_N$, virtual functions are only known through their moments, no functional space is introduced. This is equivalent to using Mimetic Finite Differences in such elements~\cite{MFD}, and classical Virtual Elements in the other ones. This would make the theoretical analysis more complicate. On the other hand, as a combination of two reliable methods, it is reasonable to expect the usual level of accuracy. The numerical results indicate a good performance.
\end{remark}

\section{Numerical Experiments}\label{sec:numerical}

\noindent
In this section, we present some numerical experiments to validate the proposed method. We measure the following relative errors in broken $H^1(\Omega)$ semi-norm and in the $L^2(\Omega)$ norm, respectively
\begin{align*}
  &e_1 = \frac{||\nabla u - \Pi_{k-1}^{0}\nabla u_h||_{0,h}}{|u|_1}, & &e_0 = \frac{||u - \Pi_k^{\nabla}u_h||_{0}}{||u||_0}.
\end{align*}
As described in the previous sections, we define the stabilization bilinear form as a suitably scaled euclidean scalar product of the vectors of degrees of freedom.

Some geometrical data of the meshes employed in the tests will be provided below. We use the following notation: $N_{\mathcal P}$, number of polygons (in 2D) or polyhedra (in 3D); $N_{\mathcal F}$, number of faces (in 3D), $N_{\mathcal E}$, number of edges; $N_{\mathcal V}$, number of vertices; $h^\text{min} = \text{min}_{\mathcal P\in\mathcal T_h}h_\mathcal{P}^\text{min}$ with $h_\mathcal{P}^\text{min}$ being the minimum distance between any pair of vertices of $\mathcal P$; $\bar h = \frac{1}{N_P}\sum_{\mathcal P\in\mathcal T_h}h_\mathcal{P}$, average mesh size.

\subsection{Patch and convergence tests with discontinuous diffusion in two dimensions}\label{test2d-k1k2}

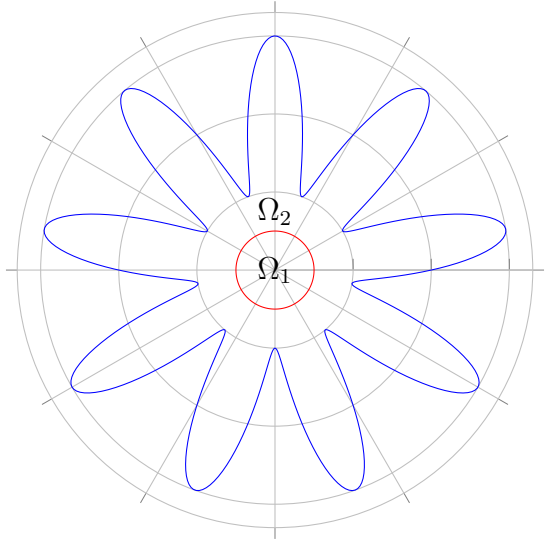
\begin{figure}
    \centering
    \begin{tikzpicture}
        \begin{polaraxis}[xticklabel=\empty, yticklabel=\empty, axis line style={color=lightgray}]
            \addplot+[mark=none, domain=0:360, samples=500] {2+sin(9*x)};
            \addplot+[mark=none, domain=0:360, samples=500] {0.5};
            \node at (0,0) {$\Omega_1$};
            \node at (90,0.75) {$\Omega_2$};
        \end{polaraxis}        
    \end{tikzpicture}
    \caption{Computational domain of test~\ref{test2d-k1k2}.}
    \label{fig:enter-label}
\end{figure}

We start by considering a problem with discontinuous diffusion in two dimensions. Let $\Omega\subset\mathbb R^2$ be the region bounded by the polar curve
\begin{equation*}
  r(\theta) = 2 + \sin(9\theta), \theta\in[0, 2\pi),
\end{equation*}
split into a smaller disk $\Omega_1$ centered at $(0,0)$ with radius $1/2$ and the complement $\Omega_2 = \Omega\setminus\Omega_1$. We consider problem~\eqref{eq:model} with:
\begin{align*}
  &\kappa_{|\Omega_1} = \kappa_1 = 0.5, \kappa_{|\Omega_2} = \kappa_2 = 10,\\
  &\Gamma_D = \Set{(x,y)\in\partial\Omega | y > 0}, \Gamma_N = \Set{(x,y)\in\partial\Omega | y < 0},
\end{align*}
and choose right hand side and boundary data $g_D, g_N$ such that the exact solution is:
\begin{align*}
  &u_1(x, y) = -\frac{r^2}{\kappa_1} + \frac{f_2-f_1}{\kappa_2}\sin(3.5\pi) + \frac{1}{4}\left(\frac{1}{\kappa_1} - \frac{1}{\kappa_2}\right) + \frac{f_2}{\kappa_2}\quad\text{in }\Omega_1,\\
  &u_2(x, y) = -\frac{r^2}{\kappa_2} + \frac{f_2-f_1}{\kappa_2}\sin(7\pi r) + \frac{f_2}{\kappa_2}\quad\text{in }\Omega_2,
\end{align*}
with $r = \sqrt{x^2+y^2}$. Since $\kappa_1\neq\kappa_2$, the solution has a discontinuous gradient along the circle of radius $1/2$.

\subsubsection{Patch test}\label{test2d-k1k2-patch}

\begin{figure}
    \centering
    \subfloat{\includegraphics[width=0.45\textwidth, valign=c]{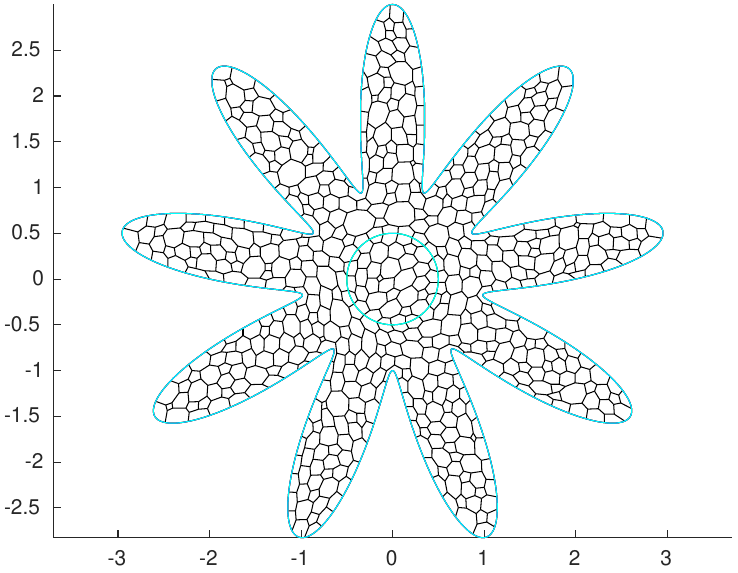}}
    \subfloat{\includegraphics[width=0.55\textwidth, valign=c]{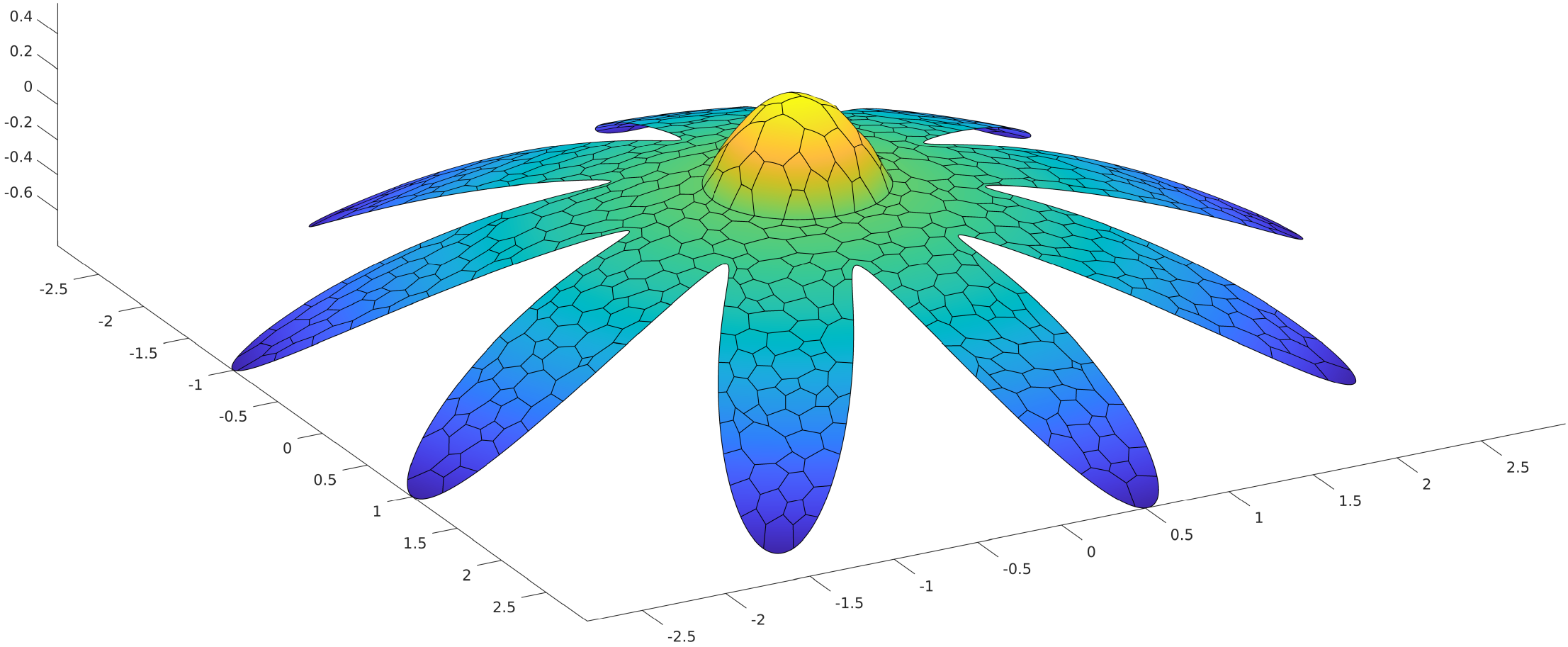}}
    \caption{Test case~\ref{test2d-k1k2-patch}, patch test: (left) mesh employed; (right) solution computed on the chosen mesh.}
    \label{fig:test2d-k1k2-patchtest}
\end{figure}

First, we take $f_1 = f_2 = 1$, so that the solution in each subregion is a paraboloid. We check that the VEM solution $u_h$ satisfies the patch test. Fig.~\ref{fig:test2d-k1k2-patchtest} shows the mesh used and the solution obtained for $k = 2$. Curved edges are displayed in cyan. The error estimators are close to machine-precision ($e_1 = \num{5.833854e-14}$ and $e_0 = \num{4.498005e-14}$), meaning that we recover the exact solution. Since the diffusion is discontinuous across the internal curved edges, the degrees of freedom defining the value on such edges must be stabilized only once (see Remark~\ref{remark-Pcoercivity}). The estimated errors obtained by stabilizing them twice, again for $k = 2$, are: $e_1 = \num{2.359148e-02}$ and $e_0 = \num{3.807228e-02}$.

\subsubsection{Convergence test}\label{test2d-k1k2-convergence}

\begin{table}
  \centering
  \caption{Data for the meshes used for test~\ref{test2d-k1k2-convergence}.}
  \label{tab:test2d-k1k2-convergence}
  \begin{tabular}{
      c
      c
      c
      c
      S[table-format=1.{\roundPrecision}e-1]
      S[table-format=1.{\roundPrecision}e-1]
      S[table-format=1.{\roundPrecision}e-1]
    }
    \toprule
        {Mesh} & {$N_{\mathcal P}$} & {$N_\mathcal{E}$} & {$N_\mathcal{V}$} & {$h$} & {$h^{\text{min}}$} & {$\overline h$}\\
        \midrule
flower$_1$ & 632 & 1873 & 1242 & 4.042816e-01 & 2.559935e-02 & 2.158110e-01\\
flower$_2$ & 2413 & 7197 & 4785 & 1.984354e-01 & 1.543251e-02 & 1.060722e-01\\
flower$_3$ & 9249 & 27666 & 18418 & 9.252595e-02 & 3.840311e-03 & 5.334973e-02\\
flower$_4$ & 35933 & 107640 & 71708 & 5.098972e-02 & 1.989369e-03 & 2.685243e-02\\
flower$_5$ & 142316 & 426633 & 284318 & 2.590563e-02 & 8.957788e-04 & 1.342535e-02\\
        \bottomrule
  \end{tabular}
\end{table}

We take $f_1 = 1, f_2 = 2$, so that the exact solution is not a quadratic polynomial in the external region $\Omega_2$ anymore. We use meshes made of non structured polygonal elements, like the one shown in Fig.~\ref{fig:test2d-k1k2-patchtest}. Geometrical data are listed in Table~\ref{tab:test2d-k1k2-convergence}. Convergence curves for $k = 1, \dots, 4$ in $H^1$ and in $L^2$ are shown in Fig~\ref{fig:test2d-k1k2-convergence}. Slopes are as expected.

\begin{figure}
  \centering
  \begin{tabular}{rl}
    \begin{tikzpicture}[baseline, trim axis left]
      \begin{loglogaxis}
	[ mark size=4pt, grid=major, small,
	  xlabel={Average mesh size $\overline h$},
        legend to name=test2d-k1k2-H1err,
        legend columns=-1,
        xtick=data,
        xticklabel={\pgfmathparse{exp(\tick)}\pgfmathprintnumber{\pgfmathresult}},
        xticklabel style={
        /pgf/number format/.cd,fixed,precision=2 },
	title={$H^1$ errors ($e_1$)}, width=0.5\textwidth ]
        \addplot[color=red,mark=x] coordinates {
          (0.215811,0.976044)
          (0.106072,0.565226)
          (0.0533497,0.270889)
          (0.0268524,0.133702)
          (0.0134254,0.0664513)
        };
        \node [fill=yellow,draw=black,anchor=center,font=\tiny] at (0.1513,0.742755) {$0.77$};
        \node [fill=yellow,draw=black,anchor=center,font=\tiny] at (0.0752258,0.391297) {$1.07$};
        \node [fill=yellow,draw=black,anchor=center,font=\tiny] at (0.0378493,0.190311) {$1.03$};
        \node [fill=yellow,draw=black,anchor=center,font=\tiny] at (0.0189869,0.0942586) {$1.01$};
        \addlegendentry{$k = 1$;}
        \addplot[color=green,mark=+] coordinates {
          (0.215811,0.547868)
          (0.106072,0.156099)
          (0.0533497,0.0399005)
          (0.0268524,0.0100281)
          (0.0134254,0.00250313)
        };
        \node [fill=yellow,draw=black,anchor=center,font=\tiny] at (0.1513,0.292441) {$1.77$};
        \node [fill=yellow,draw=black,anchor=center,font=\tiny] at (0.0752258,0.0789203) {$1.98$};
        \node [fill=yellow,draw=black,anchor=center,font=\tiny] at (0.0378493,0.0200032) {$2.01$};
        \node [fill=yellow,draw=black,anchor=center,font=\tiny] at (0.0189869,0.00501016) {$2.00$};
        \addlegendentry{$k = 2$;}
        \addplot[color=blue,mark=o] coordinates {
          (0.215811,0.237028)
          (0.106072,0.0356538)
          (0.0533497,0.00427865)
          (0.0268524,0.000552234)
          (0.0134254,6.82178e-05)
        };
        \node [fill=yellow,draw=black,anchor=center,font=\tiny] at (0.1513,0.091929) {$2.67$};
        \node [fill=yellow,draw=black,anchor=center,font=\tiny] at (0.0752258,0.0123511) {$3.09$};
        \node [fill=yellow,draw=black,anchor=center,font=\tiny] at (0.0378493,0.00153715) {$2.98$};
        \node [fill=yellow,draw=black,anchor=center,font=\tiny] at (0.0189869,0.000194093) {$3.02$};
        \addlegendentry{$k = 3$;}
        \addplot[color=orange,mark=square] coordinates {
          (0.215811,0.0872111)
          (0.106072,0.00592854)
          (0.0533497,0.000380714)
          (0.0268524,2.45042e-05)
          (0.0134254,1.52066e-06)
        };
        \node [fill=yellow,draw=black,anchor=center,font=\tiny] at (0.1513,0.0227384) {$3.79$};
        \node [fill=yellow,draw=black,anchor=center,font=\tiny] at (0.0752258,0.00150236) {$3.99$};
        \node [fill=yellow,draw=black,anchor=center,font=\tiny] at (0.0378493,9.65874e-05) {$4.00$};
        \node [fill=yellow,draw=black,anchor=center,font=\tiny] at (0.0189869,6.10431e-06) {$4.01$};
        \addlegendentry{$k = 4$}  
      \end{loglogaxis}
    \end{tikzpicture}
    
    &
    
    \begin{tikzpicture}[baseline, trim axis right]
      \begin{loglogaxis}
	[ mark size=4pt, grid=major, small,
	  xlabel={Average mesh size $\overline h$},
              xtick=data,
        xticklabel={\pgfmathparse{exp(\tick)}\pgfmathprintnumber{\pgfmathresult}},
        xticklabel style={
        /pgf/number format/.cd,fixed,precision=2 },
	title={$L^2$ errors ($e_0$)}, width=0.5\textwidth ]
        \addplot[color=red,mark=x] coordinates {
          (0.215811,0.471517)
          (0.106072,0.109698)
          (0.0533497,0.0307913)
          (0.0268524,0.00773271)
          (0.0134254,0.00197072)
        };
        \node [fill=yellow,draw=black,anchor=center,font=\tiny] at (0.1513,0.22743) {$2.05$};
        \node [fill=yellow,draw=black,anchor=center,font=\tiny] at (0.0752258,0.0581184) {$1.85$};
        \node [fill=yellow,draw=black,anchor=center,font=\tiny] at (0.0378493,0.0154305) {$2.01$};
        \node [fill=yellow,draw=black,anchor=center,font=\tiny] at (0.0189869,0.00390372) {$1.97$};
        \addplot[color=green,mark=+] coordinates {
          (0.215811,0.120076)
          (0.106072,0.0155614)
          (0.0533497,0.00166363)
          (0.0268524,0.000201442)
          (0.0134254,2.4609e-05)
        };
        \node [fill=yellow,draw=black,anchor=center,font=\tiny] at (0.1513,0.0432269) {$2.88$};
        \node [fill=yellow,draw=black,anchor=center,font=\tiny] at (0.0752258,0.00508807) {$3.25$};
        \node [fill=yellow,draw=black,anchor=center,font=\tiny] at (0.0378493,0.000578899) {$3.08$};
        \node [fill=yellow,draw=black,anchor=center,font=\tiny] at (0.0189869,7.0408e-05) {$3.03$};
        \addplot[color=blue,mark=o] coordinates {
          (0.215811,0.0421723)
          (0.106072,0.00243765)
          (0.0533497,0.000133943)
          (0.0268524,8.78073e-06)
          (0.0134254,5.32707e-07)
        };
        \node [fill=yellow,draw=black,anchor=center,font=\tiny] at (0.1513,0.0101391) {$4.01$};
        \node [fill=yellow,draw=black,anchor=center,font=\tiny] at (0.0752258,0.000571407) {$4.22$};
        \node [fill=yellow,draw=black,anchor=center,font=\tiny] at (0.0378493,3.42946e-05) {$3.97$};
        \node [fill=yellow,draw=black,anchor=center,font=\tiny] at (0.0189869,2.16277e-06) {$4.04$};
        \addplot[color=orange,mark=square] coordinates {
          (0.215811,0.013579)
          (0.106072,0.000407944)
          (0.0533497,9.94273e-06)
          (0.0268524,3.29039e-07)
          (0.0134254,1.00958e-08)
        };
        \node [fill=yellow,draw=black,anchor=center,font=\tiny] at (0.1513,0.00235361) {$4.93$};
        \node [fill=yellow,draw=black,anchor=center,font=\tiny] at (0.0752258,6.36874e-05) {$5.40$};
        \node [fill=yellow,draw=black,anchor=center,font=\tiny] at (0.0378493,1.80874e-06) {$4.96$};
        \node [fill=yellow,draw=black,anchor=center,font=\tiny] at (0.0189869,5.76359e-08) {$5.03$};
      \end{loglogaxis}
    \end{tikzpicture}
  \end{tabular}
\ref{test2d-k1k2-H1err}
  \caption{Convergence plots for Test~\ref{test2d-k1k2-convergence}.}
  \label{fig:test2d-k1k2-convergence}
\end{figure}
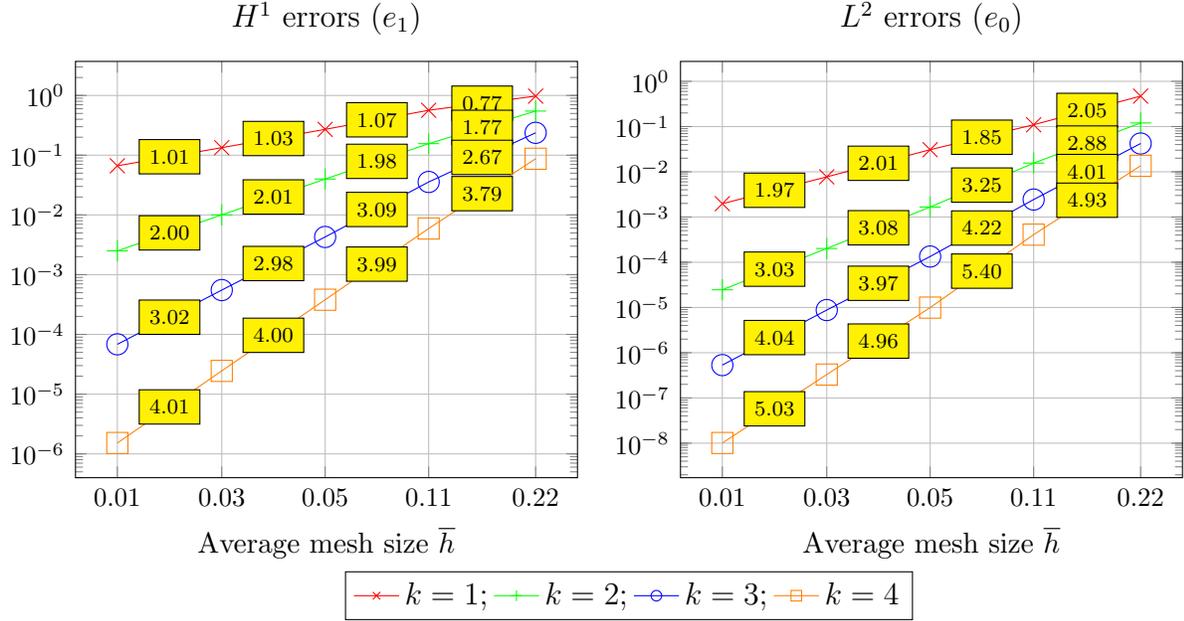

\subsection{Patch and convergence tests with discontinuous diffusion in three dimensions}\label{test3d-ball}

Now we consider a problem with discontinuous diffusion in three dimensions. Let $\Omega\subset\mathbb R^3$ be the unit ball, split into a smaller ball $\Omega_1$ centered at the origin with radius $1/2$ and the complement $\Omega_2 = \Omega\setminus\Omega_1$. We consider problem~\eqref{eq:model} with:
\begin{align*}
  &\kappa_{|\Omega_1} = \kappa_1 = 0.5, \kappa_{|\Omega_2} = \kappa_2 = 10,\\
  &\Gamma_D = \Set{(x,y,z)\in\partial\Omega | z > 0}, \Gamma_N = \Set{(x,y,z)\in\partial\Omega | z < 0},
\end{align*}
and choose right hand side and boundary data $g_D, g_N$ such that the exact solution is:
\begin{align*}
  &u_1(x, y, z) = -\frac{r^2}{\kappa_1} + \frac{f_2-f_1}{\kappa_2}\sin(0.5\pi) + \frac{1}{4}\left(\frac{1}{\kappa_1} - \frac{1}{\kappa_2}\right) + \frac{f_2}{\kappa_2}\quad\text{in }\Omega_1,\\
  &u_2(x, y, z) = -\frac{r^2}{\kappa_2} + \frac{f_2-f_1}{\kappa_2}\sin(\pi r) + \frac{f_2}{\kappa_2}\quad\text{in }\Omega_2,
\end{align*}
with $r = \sqrt{x^2+y^2+z^2}$. Since $\kappa_1\neq\kappa_2$, the solution has a discontinuous gradient across the sphere of radius $1/2$.

\subsubsection{Patch test}\label{test3d-ball-patch}

\begin{figure}
    \centering
    \subfloat{\includegraphics[width=0.5\textwidth, valign=c]{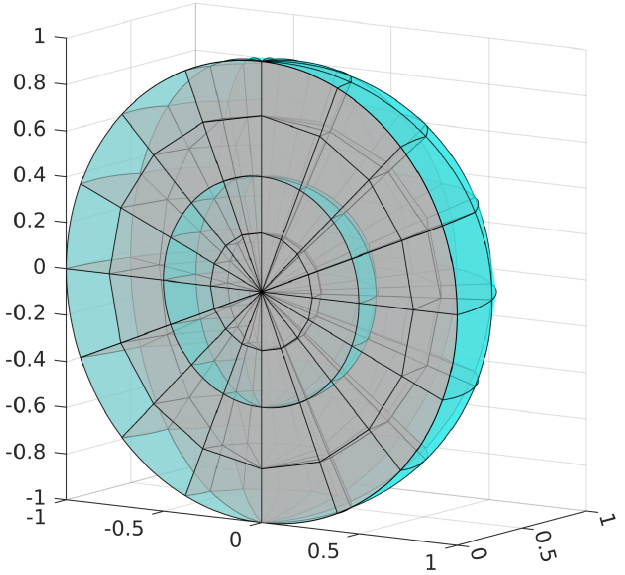}}\quad
    \subfloat{\includegraphics[width=0.4\textwidth, valign=c]{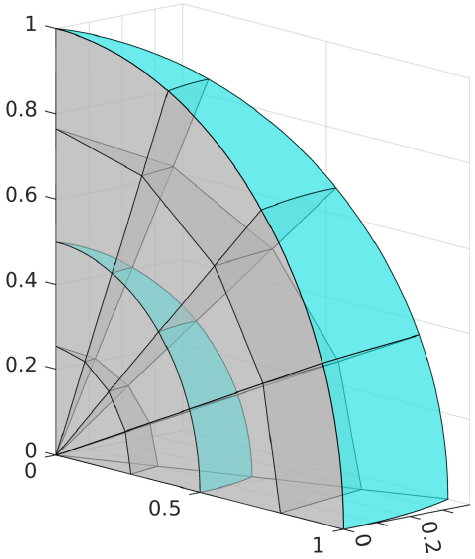}}
    \caption{Sections of the mesh used for test case~\ref{test3d-ball-patch}. Curved faces are highlighted in cyan.}
    \label{fig:test3d-ball-patch}
\end{figure}

By taking $f_1 = f_2 = 1$, we let the exact solution be a quadratic polynomial in each subregion. Fig.~\ref{fig:test3d-ball-patch} shows the mesh used. Curved faces are highlighted in cyan. For $k = 2$, by stabilizing the dofs defining the value on curved faces only once, we get $e_1 = \num{1.748175e-13}$ and $e_0 = \num{9.970228e-14}$, that is we recover the exact solution. Instead, always for $k = 2$, the error estimators obtained by stabilizing twice are $e_1 = \num{5.395074e-02}$ and $e_0 = \num{7.362497e-03}$.

\subsubsection{Convergence test}\label{test3d-ball-convergence}

\begin{table}
  \centering
  \caption{Data for the meshes used for test~\ref{test3d-ball-convergence}.}
  \label{tab:test3d-ball-convergence}
  \begin{tabular}{
      c
      c
      c
      c
      c
      S[table-format=1.{\roundPrecision}e-1]
      S[table-format=1.{\roundPrecision}e-1]
      S[table-format=1.{\roundPrecision}e-1]
    }
    \toprule
        {Mesh} & {$N_{\mathcal P}$} & {$N_{\mathcal F}$} & {$N_\mathcal{E}$} & {$N_\mathcal{V}$} & {$h$} & {$h^{\text{min}}$} & {$\overline h$}\\
        \midrule
ball$_{1}$ & 512 & 1472 & 1416 & 457 & 5.367641e-01 & 3.880588e-02 & 3.582102e-01\\
ball$_{2}$ & 4096 & 12032 & 11792 & 3857 & 2.837977e-01 & 4.826923e-03 & 1.805581e-01\\
ball$_{3}$ & 32768 & 97280 & 96288 & 31777 & 1.471917e-01 & 6.026351e-04 & 9.042229e-02\\
ball$_{4}$ & 262144 & 782336 & 778304 & 258113 & 7.487764e-02 & 7.530667e-05 & 4.522186e-02\\
        \bottomrule
  \end{tabular}
\end{table}

We set $f_1 = 1, f_2 = 2$: the exact solution is not a quadratic polynomial in the external region $\Omega_2$ anymore. We take meshes like the one shown in Fig.~\ref{fig:test3d-ball-patch}. Geometrical data are listed in Table~\ref{tab:test3d-ball-convergence}. Convergence curves for $k = 1, 2, 3$ in $H^1$ and in $L^2$ are shown in Fig~\ref{fig:test3d-ball-convergence}. Slopes are as expected.

\begin{figure}
  \centering
  \begin{tabular}{rl}
    \begin{tikzpicture}[baseline, trim axis left]
      \begin{loglogaxis}
	[ mark size=4pt, grid=major, small,
	  xlabel={Average mesh size $\overline h$},
        legend to name=test3d-ball-H1err,
        legend columns=-1,
        xtick=data,
        xticklabel={\pgfmathparse{exp(\tick)}\pgfmathprintnumber{\pgfmathresult}},
        xticklabel style={
        /pgf/number format/.cd,fixed,precision=2 },
	title={$H^1$ errors ($e_1$)}, width=0.5\textwidth ]
        \addplot[color=red,mark=x] coordinates {
          (0.35821,0.281964)
          (0.180558,0.122962)
          (0.0904223,0.057876)
          (0.0452219,0.0283934)
        };
        \node [fill=yellow,draw=black,anchor=center,font=\tiny] at (0.254318,0.186201) {$1.21$};
        \node [fill=yellow,draw=black,anchor=center,font=\tiny] at (0.127775,0.0843596) {$1.09$};
        \node [fill=yellow,draw=black,anchor=center,font=\tiny] at (0.0639458,0.0405375) {$1.03$};
        \addlegendentry{$k = 1$;}
        \addplot[color=green,mark=+] coordinates {
          (0.35821,0.0110099)
          (0.180558,0.00308614)
          (0.0904223,0.000768059)
          (0.0452219,0.000182219)
        };
        \node [fill=yellow,draw=black,anchor=center,font=\tiny] at (0.254318,0.00582906) {$1.86$};
        \node [fill=yellow,draw=black,anchor=center,font=\tiny] at (0.127775,0.00153959) {$2.01$};
        \node [fill=yellow,draw=black,anchor=center,font=\tiny] at (0.0639458,0.000374106) {$2.08$};
        \addlegendentry{$k = 2$;}
        \addplot[color=blue,mark=o] coordinates {
          (0.35821,0.00228482)
          (0.180558,0.000296661)
          (0.0904223,3.73072e-05)
          (0.0452219,4.65551e-06)
        };
        \node [fill=yellow,draw=black,anchor=center,font=\tiny] at (0.254318,0.000823297) {$2.98$};
        \node [fill=yellow,draw=black,anchor=center,font=\tiny] at (0.127775,0.000105203) {$3.00$};
        \node [fill=yellow,draw=black,anchor=center,font=\tiny] at (0.0639458,1.31789e-05) {$3.00$};
        \addlegendentry{$k = 3$}
    \end{loglogaxis}
    \end{tikzpicture}
    
    &
    
    \begin{tikzpicture}[baseline, trim axis right]
      \begin{loglogaxis}
	[ mark size=4pt, grid=major, small,
	  xlabel={Average mesh size $\overline h$},
              xtick=data,
        xticklabel={\pgfmathparse{exp(\tick)}\pgfmathprintnumber{\pgfmathresult}},
        xticklabel style={
        /pgf/number format/.cd,fixed,precision=2 },
	title={$L^2$ errors ($e_0$)}, width=0.5\textwidth ]
        \addplot[color=red,mark=x] coordinates {
          (0.35821,0.0866812)
          (0.180558,0.0258842)
          (0.0904223,0.00654725)
          (0.0452219,0.00162332)
        };
        \node [fill=yellow,draw=black,anchor=center,font=\tiny] at (0.254318,0.0473674) {$1.76$};
        \node [fill=yellow,draw=black,anchor=center,font=\tiny] at (0.127775,0.0130181) {$1.99$};
        \node [fill=yellow,draw=black,anchor=center,font=\tiny] at (0.0639458,0.0032601) {$2.01$};
        \addlegendentry{$k = 1$}
        \addplot[color=green,mark=+] coordinates {
          (0.35821,0.00142705)
          (0.180558,0.000153126)
          (0.0904223,1.94079e-05)
          (0.0452219,2.42799e-06)
        };
        \node [fill=yellow,draw=black,anchor=center,font=\tiny] at (0.254318,0.000467461) {$3.26$};
        \node [fill=yellow,draw=black,anchor=center,font=\tiny] at (0.127775,5.45148e-05) {$2.99$};
        \node [fill=yellow,draw=black,anchor=center,font=\tiny] at (0.0639458,6.86456e-06) {$3.00$};
        \addlegendentry{$k = 2$}
        \addplot[color=blue,mark=o] coordinates {
          (0.35821,0.000159376)
          (0.180558,1.15406e-05)
          (0.0904223,7.17218e-07)
          (0.0452219,4.25188e-08)
        };
        \node [fill=yellow,draw=black,anchor=center,font=\tiny] at (0.254318,4.28871e-05) {$3.83$};
        \node [fill=yellow,draw=black,anchor=center,font=\tiny] at (0.127775,2.877e-06) {$4.02$};
        \node [fill=yellow,draw=black,anchor=center,font=\tiny] at (0.0639458,1.74629e-07) {$4.08$};
        \addlegendentry{$k = 3$}
              \end{loglogaxis}
    \end{tikzpicture}
  \end{tabular}
\ref{test3d-ball-H1err}
  \caption{Convergence plots for Test~\ref{test3d-ball-convergence}.}
  \label{fig:test3d-ball-convergence}
\end{figure}
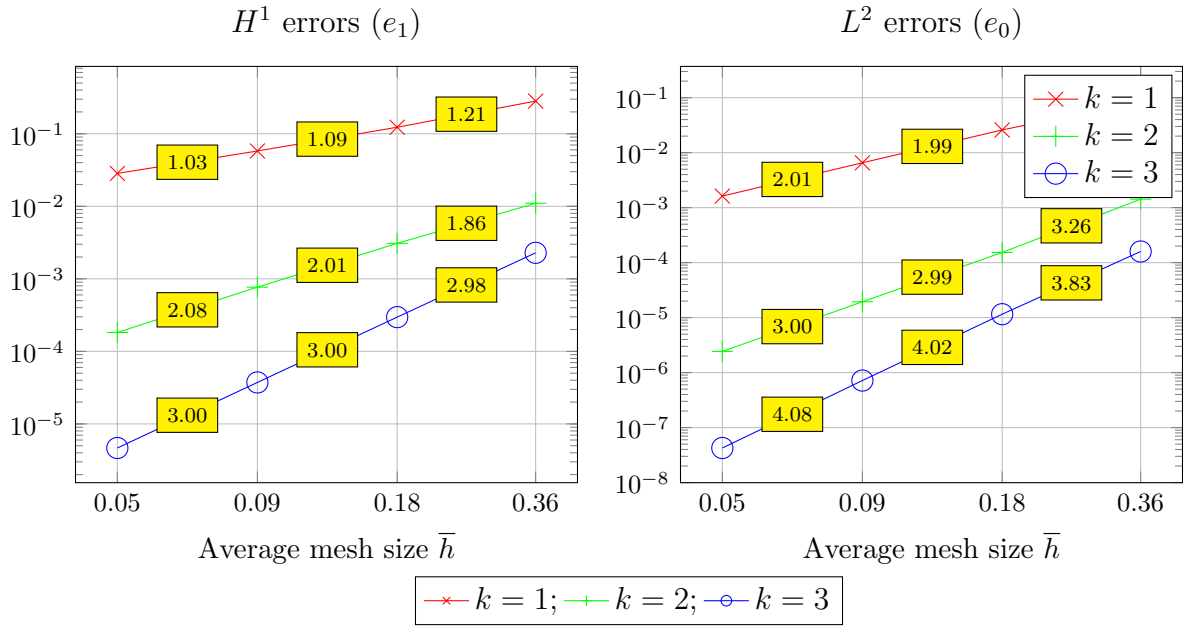

\subsection{Convergence test on a toroidal domain}\label{test3d-torus-convergence}

\begin{figure}
    \centering
    \subfloat{\includegraphics[width=0.5\textwidth, valign=c]{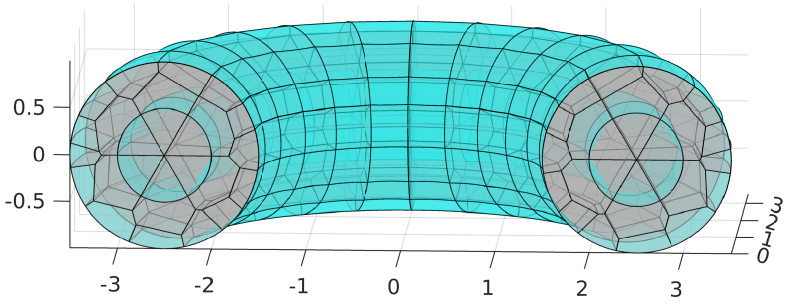}}\quad
    \subfloat{\includegraphics[width=0.4\textwidth, valign=c]{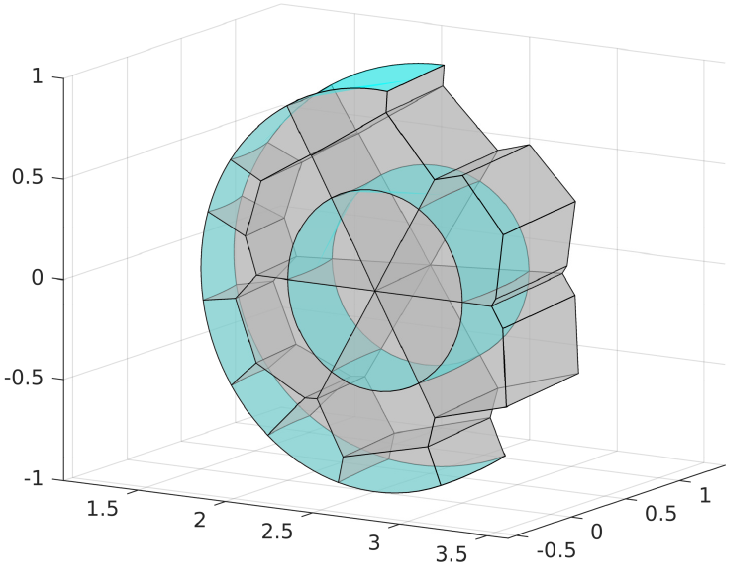}}
    \caption{Sections of the mesh used for test case~\ref{test3d-torus-convergence}. Curved faces are highlighted in cyan.}
    \label{fig:test3d-torus-mesh}
\end{figure}

\begin{table}
  \centering
  \caption{Data for the meshes used for test~\ref{test3d-torus-convergence}.}
  \label{tab:test3d-torus-convergence}
  \begin{tabular}{
      c
      c
      c
      c
      c
      S[table-format=1.{\roundPrecision}e-1]
      S[table-format=1.{\roundPrecision}e-1]
      S[table-format=1.{\roundPrecision}e-1]
    }
    \toprule
        {Mesh} & {$N_{\mathcal P}$} & {$N_{\mathcal F}$} & {$N_\mathcal{E}$} & {$N_\mathcal{V}$} & {$h$} & {$h^{\text{min}}$} & {$\overline h$}\\
        \midrule
torus$_{1}$ & 624 & 2208 & 2568 & 984 & 1.035595e+00 & 6.949816e-03 & 8.537336e-01\\
torus$_{2}$ & 3552 & 13536 & 16464 & 6480 & 5.458442e-01 & 3.540881e-03 & 4.508194e-01\\
torus$_{3}$ & 23424 & 91392 & 112608 & 44640 & 2.729370e-01 & 3.190379e-03 & 2.261588e-01\\
torus$_{4}$ & 154176 & 608064 & 753792 & 299904 & 1.419995e-01 & 1.475534e-03 & 1.162349e-01\\
        \bottomrule
  \end{tabular}
\end{table}

With this test, we investigate the robustness of the proposed method on a non-convex domain with non constant curvature, meshed with non-structured curved polytopes. Let $\Omega\subset\mathbb R^3$ be the region bounded by the torus whose parametric equations are
\begin{align*}
  x &= (R + r \cos v)\cos u, & y &= (R + r\cos v)\sin u, & z &= r\sin v,
\end{align*}
with $u, v \in [0, 2\pi)$, $R = 2.5$, $r = 1$, split into a smaller region $\Omega_1$ bounded by a toroidal surface with $r = 0.5$ and the complement $\Omega\setminus\Omega_1$. We consider problem~\eqref{eq:model} with:
\begin{align*}
  &\kappa_{|\Omega_1} = \kappa_1 = 1, \kappa_{|\Omega_2} = \kappa_2 = 10,\\
  &\Gamma_D = \Set{(x,y,z)\in\partial\Omega | y > 0}, \Gamma_N = \Set{(x,y,z)\in\partial\Omega | y < 0},
\end{align*}
and choose right hand side and boundary data $g_D, g_N$ such that the exact solution is:
\begin{align*}
  &u_1(x, y, z) = \frac{1}{\kappa_1}e^{-\rho^2}\cos u & &\text{in }\Omega_1,\\
  &u_2(x, y, z) = \left(\frac{1}{\kappa_2}e^{-\rho^2} + \left(\frac{1}{\kappa_1}-\frac{1}{\kappa_2}\right)e^{-0.5^2}\right)\cos u & &\text{in }\Omega_2,
\end{align*}
with $\rho = \sqrt{(\sqrt{x^2+y^2} - R)^2 + z^2}$. The solution has a discontinuous gradient across $\partial\Omega_1$. We consider meshes like the one shown in Fig.~\ref{fig:test3d-torus-mesh}. Geometrical data are listed in Table~\ref{tab:test3d-torus-convergence}. Convergence curves for $k = 1, 2, 3$ in $H^1$ and $L^2$ are shown in Fig.~\ref{fig:test3d-torus-convergence}. Slopes are as expected.

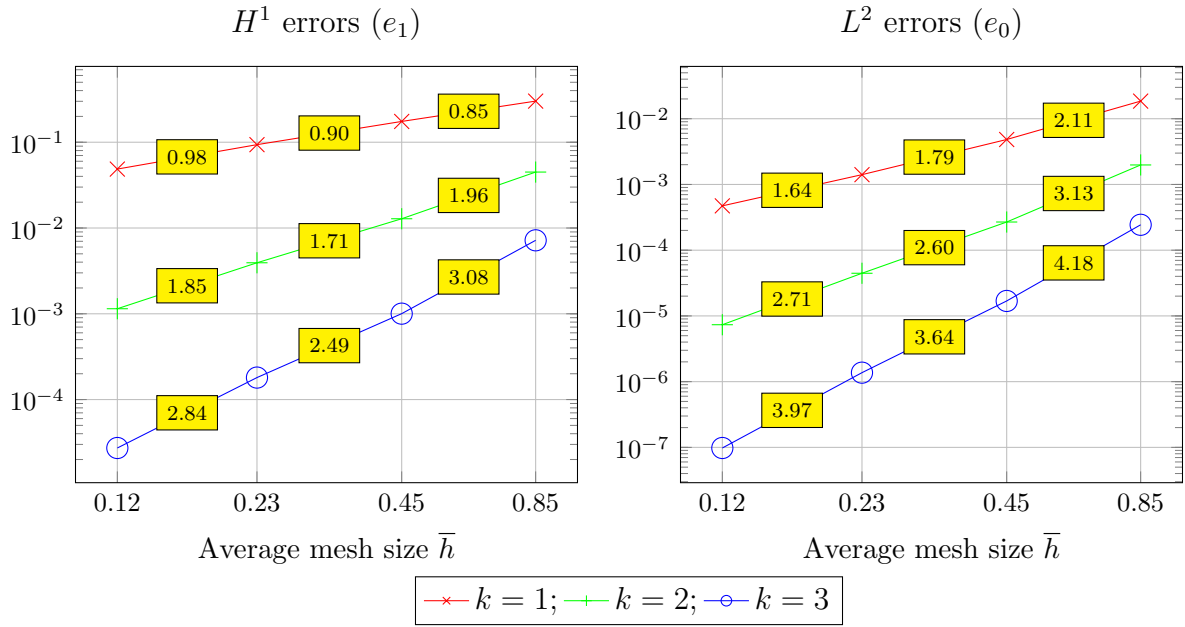
\begin{figure}
  \centering
  \begin{tabular}{rl}
    \begin{tikzpicture}[baseline, trim axis left]
      \begin{loglogaxis}
	[ mark size=4pt, grid=major, small,
	  xlabel={Average mesh size $\overline h$},
        legend to name=test3d-torus-H1err,
        legend columns=-1,
        xtick=data,
        xticklabel={\pgfmathparse{exp(\tick)}\pgfmathprintnumber{\pgfmathresult}},
        xticklabel style={
        /pgf/number format/.cd,fixed,precision=2 },
	title={$H^1$ errors ($e_1$)}, width=0.5\textwidth ]
        \addplot[color=red,mark=x] coordinates {
          (0.853734,0.301617)
          (0.450819,0.174734)
          (0.226159,0.0936936)
          (0.116235,0.0486498)
        };
        \node [fill=yellow,draw=black,anchor=center,font=\tiny] at (0.620387,0.22957) {$0.85$};
        \node [fill=yellow,draw=black,anchor=center,font=\tiny] at (0.319307,0.127951) {$0.90$};
        \node [fill=yellow,draw=black,anchor=center,font=\tiny] at (0.162134,0.0675143) {$0.98$};
        \addlegendentry{$k = 1$;}
        \addplot[color=green,mark=+] coordinates {
          (0.853734,0.0448944)
          (0.450819,0.0128217)
          (0.226159,0.00393364)
          (0.116235,0.00114527)
        };
        \node [fill=yellow,draw=black,anchor=center,font=\tiny] at (0.620387,0.0239922) {$1.96$};
        \node [fill=yellow,draw=black,anchor=center,font=\tiny] at (0.319307,0.00710184) {$1.71$};
        \node [fill=yellow,draw=black,anchor=center,font=\tiny] at (0.162134,0.00212252) {$1.85$};
        \addlegendentry{$k = 2$;}
        \addplot[color=blue,mark=o] coordinates {
          (0.853734,0.00718693)
          (0.450819,0.00100358)
          (0.226159,0.000180463)
          (0.116235,2.72465e-05)
        };
        \node [fill=yellow,draw=black,anchor=center,font=\tiny] at (0.620387,0.00268564) {$3.08$};
        \node [fill=yellow,draw=black,anchor=center,font=\tiny] at (0.319307,0.000425569) {$2.49$};
        \node [fill=yellow,draw=black,anchor=center,font=\tiny] at (0.162134,7.01213e-05) {$2.84$};
        \addlegendentry{$k = 3$}
    \end{loglogaxis}
    \end{tikzpicture}
    
    &
    
    \begin{tikzpicture}[baseline, trim axis right]
      \begin{loglogaxis}
	[ mark size=4pt, grid=major, small,
	  xlabel={Average mesh size $\overline h$},
        xtick=data,
        xticklabel={\pgfmathparse{exp(\tick)}\pgfmathprintnumber{\pgfmathresult}},
        xticklabel style={
        /pgf/number format/.cd,fixed,precision=2 },
	title={$L^2$ errors ($e_0$)}, width=0.5\textwidth ]
        \addplot[color=red,mark=x] coordinates {
          (0.853734,0.0185443)
          (0.450819,0.00483167)
          (0.226159,0.00140959)
          (0.116235,0.000472404)
        };
        \node [fill=yellow,draw=black,anchor=center,font=\tiny] at (0.620387,0.00946573) {$2.11$};
        \node [fill=yellow,draw=black,anchor=center,font=\tiny] at (0.319307,0.00260972) {$1.79$};
        \node [fill=yellow,draw=black,anchor=center,font=\tiny] at (0.162134,0.000816024) {$1.64$};
        \addplot[color=green,mark=+] coordinates {
          (0.853734,0.00198096)
          (0.450819,0.000268336)
          (0.226159,4.45627e-05)
          (0.116235,7.34631e-06)
        };
        \node [fill=yellow,draw=black,anchor=center,font=\tiny] at (0.620387,0.000729085) {$3.13$};
        \node [fill=yellow,draw=black,anchor=center,font=\tiny] at (0.319307,0.000109352) {$2.60$};
        \node [fill=yellow,draw=black,anchor=center,font=\tiny] at (0.162134,1.80934e-05) {$2.71$};
        \addplot[color=blue,mark=o] coordinates {
          (0.853734,0.000243508)
          (0.450819,1.68925e-05)
          (0.226159,1.36982e-06)
          (0.116235,9.78285e-08)
        };
        \node [fill=yellow,draw=black,anchor=center,font=\tiny] at (0.620387,6.41363e-05) {$4.18$};
        \node [fill=yellow,draw=black,anchor=center,font=\tiny] at (0.319307,4.81037e-06) {$3.64$};
        \node [fill=yellow,draw=black,anchor=center,font=\tiny] at (0.162134,3.6607e-07) {$3.97$};
    \end{loglogaxis}
    \end{tikzpicture}
  \end{tabular}
\ref{test3d-torus-H1err}
  \caption{Convergence plots for Test~\ref{test3d-torus-convergence}.}
  \label{fig:test3d-torus-convergence}
\end{figure}

\subsection{Patch and convergence tests on a bubble mesh}\label{test3d-bubble}

\begin{figure}
    \centering
    \subfloat{\includegraphics[width=0.4\textwidth, valign=c]{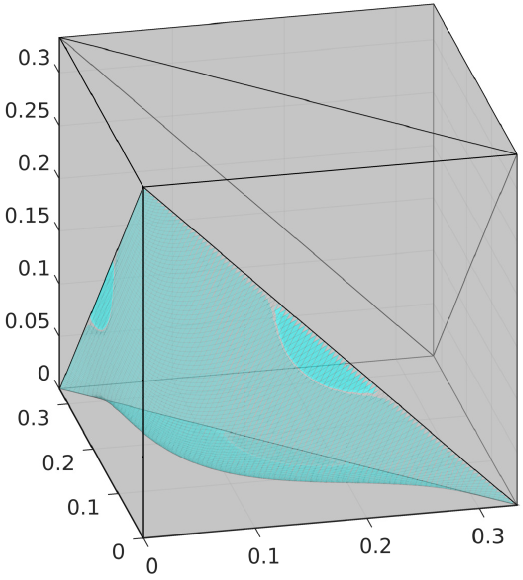}}\quad
    \subfloat{\includegraphics[width=0.5\textwidth, valign=c]{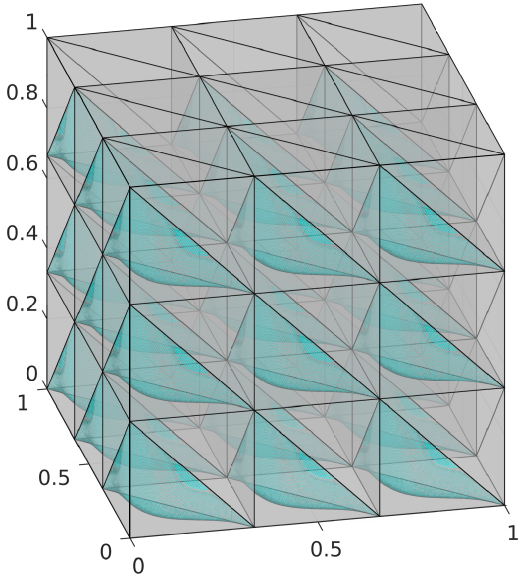}}
    \caption{(Left): a single cube with a ``bubble'' curved face separating two elements. (Right): full mesh. Curved faces are highlighted in cyan.}
    \label{fig:bubble-mesh}
\end{figure}

\noindent
For the next two tests, we consider the problem
\[
-\Delta u = f\quad\text{in }\Omega = (0,1)\times(0,1)\times(0,1),\qquad u = g_D\quad\text{on }\partial\Omega,
\]
with $f$ and $g_D$ chosen case by case. The unit cube is partitioned into small cubes of side $l$, and in turn each small cube is subdivided into two elements sharing a curved face as shown in Fig.~\ref{fig:bubble-mesh}. For each small cube $[x_i,x_i+l]\times [y_i,y_i+l]\times [z_i,z_i+l]$, the curved face is a cubic bubble with respect to the plane $\eta_i$ passing through the points $(x_i+l, y_i, z_i)$, $(x_i, y_i+l, z_i)$, and $(x_i, y_i, z_i+l)$. The distance from the tip of the bubble to the plane $\eta_i$ is proportional to $l^\beta$, with $\beta = 1$ or $\beta = 2$. For $\beta = 1$, all elements are homothetically equivalent, whereas, for $\beta = 2$, when $l$ gets smaller, the elements get flatter. Our aim is to test the robustness of the proposed method with respect to the amount of curvature of a face.

\subsubsection{Patch test}

Right hand side $f$ and boundary datum $g_D$ are chosen such that the exact solution is
\[
u(x,y,z) = -x^3 + x^2y + y^2z - xyz + y^3 - xz^2 - z^3\quad\text{ in }\Omega.
\]
We compute the error estimators in the case $k = 3$ on a mesh obtained by dividing the unit cube into $6\times 6\times 6$ small cubes. For $\beta = 1$, we get $e_1 = \num{1.763723e-11}$ and $e_0 = \num{4.775303e-12}$, whereas, for $\beta = 2$, we get $e_1 = \num{4.664029e-12}$ and $e_0 = \num{1.162185e-12}$. Thus, we recover the exact solution for both values of $\beta$.

\subsubsection{Convergence test}\label{test3d-bubble-convergence}

\begin{table}
  \centering
  \caption{Data for the meshes used for test~\ref{test3d-bubble-convergence}.}
  \label{tab:test3d-bubble-convergence}
  \begin{tabular}{
      c
      c
      c
      c
      c
      S[table-format=1.{\roundPrecision}e-1]
      S[table-format=1.{\roundPrecision}e-1]
      S[table-format=1.{\roundPrecision}e-1]
    }
    \toprule
        {Mesh} & {$N_{\mathcal P}$} & {$N_{\mathcal F}$} & {$N_\mathcal{E}$} & {$N_\mathcal{V}$} & {$h$} & {$h^{\text{min}}$} & {$\overline h$}\\
        \midrule
bubble$_{1}$ & 432 & 1728 & 1638 & 343 & 2.886751e-01 & 1.666667e-01 & 2.621887e-01\\
bubble$_{2}$ & 3456 & 12960 & 11700 & 2197 & 1.443376e-01 & 8.333333e-02 & 1.310943e-01\\
bubble$_{3}$ & 27648 & 100224 & 88200 & 15625 & 7.216878e-02 & 4.166667e-02 & 6.554717e-02\\
bubble$_{4}$ & 221184 & 787968 & 684432 & 117649 & 3.608439e-02 & 2.083333e-02 & 3.277359e-02\\
        \bottomrule
  \end{tabular}
\end{table}

\noindent
We consider the following exact solution
\[
u(x,y,z) = \cos(x+y+z)e^{xyz}\quad\text{ in }\Omega.
\]
Geometrical data for the mesh used are shown in Table~\ref{tab:test3d-bubble-convergence}. In Fig.~\ref{fig:test3d-bubble-convergence-1}, we show convergence curves for $k = 1, 2, 3$ in $H^1$ and $L^2$ for $\beta = 1$, whereas in Fig.~\ref{fig:test3d-bubble-convergence-2} we repeat the same tests for $\beta = 2$. Slopes are as expected.

\begin{figure}
  \centering
  \begin{tabular}{rl}
    \begin{tikzpicture}[baseline, trim axis left]
      \begin{loglogaxis}
	[ mark size=4pt, grid=major, small,
	  xlabel={Average mesh size $\overline h$},
        legend to name=test3d-bubble-H1err-1,
        legend columns=-1,
        xtick=data,
        xticklabel={\pgfmathparse{exp(\tick)}\pgfmathprintnumber{\pgfmathresult}},
        xticklabel style={
        /pgf/number format/.cd,fixed,precision=2 },
	title={$H^1$ errors ($e_1$)}, width=0.5\textwidth ]
\addplot[color=red,mark=x] coordinates {
(0.262189,0.117437)
(0.131094,0.0595101)
(0.0655472,0.0299413)
(0.0327736,0.0150142)
};
\node [fill=yellow,draw=black,anchor=center,font=\tiny] at (0.185395,0.0835984) {$0.98$};
\node [fill=yellow,draw=black,anchor=center,font=\tiny] at (0.0926977,0.0422115) {$0.99$};
\node [fill=yellow,draw=black,anchor=center,font=\tiny] at (0.0463489,0.0212025) {$1.00$};
\addlegendentry{$k = 1$;}
\addplot[color=green,mark=+] coordinates {
(0.262189,0.00872399)
(0.131094,0.00191097)
(0.0655472,0.000446708)
(0.0327736,0.000108084)
};
\node [fill=yellow,draw=black,anchor=center,font=\tiny] at (0.185395,0.00408305) {$2.19$};
\node [fill=yellow,draw=black,anchor=center,font=\tiny] at (0.0926977,0.000923929) {$2.10$};
\node [fill=yellow,draw=black,anchor=center,font=\tiny] at (0.0463489,0.000219731) {$2.05$};
\addlegendentry{$k = 2$;}
\addplot[color=blue,mark=o] coordinates {
(0.262189,0.000545953)
(0.131094,5.89414e-05)
(0.0655472,6.8816e-06)
(0.0327736,8.32844e-07)
};
\node [fill=yellow,draw=black,anchor=center,font=\tiny] at (0.185395,0.000179386) {$3.21$};
\node [fill=yellow,draw=black,anchor=center,font=\tiny] at (0.0926977,2.01398e-05) {$3.10$};
\node [fill=yellow,draw=black,anchor=center,font=\tiny] at (0.0463489,2.39401e-06) {$3.05$};
\addlegendentry{$k = 3$}
    \end{loglogaxis}
    \end{tikzpicture}
    
    &
    
    \begin{tikzpicture}[baseline, trim axis right]
      \begin{loglogaxis}
	[ mark size=4pt, grid=major, small,
	  xlabel={Average mesh size $\overline h$},
        xtick=data,
        xticklabel={\pgfmathparse{exp(\tick)}\pgfmathprintnumber{\pgfmathresult}},
        xticklabel style={
        /pgf/number format/.cd,fixed,precision=2 },
	title={$L^2$ errors ($e_0$)}, width=0.5\textwidth ]
\addplot[color=red,mark=x] coordinates {
(0.262189,0.0104921)
(0.131094,0.00278953)
(0.0655472,0.000732986)
(0.0327736,0.00019926)
};
\node [fill=yellow,draw=black,anchor=center,font=\tiny] at (0.185395,0.00540999) {$1.91$};
\node [fill=yellow,draw=black,anchor=center,font=\tiny] at (0.0926977,0.00142993) {$1.93$};
\node [fill=yellow,draw=black,anchor=center,font=\tiny] at (0.0463489,0.000382171) {$1.88$};
\addplot[color=green,mark=+] coordinates {
(0.262189,0.000600848)
(0.131094,6.41228e-05)
(0.0655472,7.52151e-06)
(0.0327736,9.20429e-07)
};
\node [fill=yellow,draw=black,anchor=center,font=\tiny] at (0.185395,0.000196286) {$3.23$};
\node [fill=yellow,draw=black,anchor=center,font=\tiny] at (0.0926977,2.19613e-05) {$3.09$};
\node [fill=yellow,draw=black,anchor=center,font=\tiny] at (0.0463489,2.63116e-06) {$3.03$};
\addplot[color=blue,mark=o] coordinates {
(0.262189,3.32898e-05)
(0.131094,1.62751e-06)
(0.0655472,9.10853e-08)
(0.0327736,5.41239e-09)
};
\node [fill=yellow,draw=black,anchor=center,font=\tiny] at (0.185395,7.36068e-06) {$4.35$};
\node [fill=yellow,draw=black,anchor=center,font=\tiny] at (0.0926977,3.85023e-07) {$4.16$};
\node [fill=yellow,draw=black,anchor=center,font=\tiny] at (0.0463489,2.22034e-08) {$4.07$};
    \end{loglogaxis}
    \end{tikzpicture}
  \end{tabular}
\ref{test3d-bubble-H1err-1}
  \caption{Convergence plots for Test~\ref{test3d-bubble-convergence}, case $\beta = 1$.}
  \label{fig:test3d-bubble-convergence-1}
\end{figure}

\begin{figure}
  \centering
  \begin{tabular}{rl}
    \begin{tikzpicture}[baseline, trim axis left]
      \begin{loglogaxis}
	[ mark size=4pt, grid=major, small,
	  xlabel={Average mesh size $\overline h$},
        legend to name=test3d-bubble-H1err-2,
        legend columns=-1,
        xtick=data,
        xticklabel={\pgfmathparse{exp(\tick)}\pgfmathprintnumber{\pgfmathresult}},
        xticklabel style={
        /pgf/number format/.cd,fixed,precision=2 },
	title={$H^1$ errors ($e_1$)}, width=0.5\textwidth ]
        \addplot[color=red,mark=x] coordinates {
          (0.262189,0.0631736)
          (0.131094,0.0297087)
          (0.0655472,0.0144235)
          (0.0327736,0.00710876)
        };
        \node [fill=yellow,draw=black,anchor=center,font=\tiny] at (0.185395,0.0433221) {$1.09$};
        \node [fill=yellow,draw=black,anchor=center,font=\tiny] at (0.0926977,0.0207004) {$1.04$};
        \node [fill=yellow,draw=black,anchor=center,font=\tiny] at (0.0463489,0.0101259) {$1.02$};
        \addlegendentry{$k = 1$;}
        \addplot[color=green,mark=+] coordinates {
          (0.262189,0.0073045)
          (0.131094,0.00162955)
          (0.0655472,0.00038319)
          (0.0327736,9.2858e-05)
        };
        \node [fill=yellow,draw=black,anchor=center,font=\tiny] at (0.185395,0.00345008) {$2.16$};
        \node [fill=yellow,draw=black,anchor=center,font=\tiny] at (0.0926977,0.000790208) {$2.09$};
        \node [fill=yellow,draw=black,anchor=center,font=\tiny] at (0.0463489,0.000188633) {$2.04$};
        \addlegendentry{$k = 2$;}
        \addplot[color=blue,mark=o] coordinates {
          (0.262189,0.000365217)
          (0.131094,3.89379e-05)
          (0.0655472,4.48743e-06)
          (0.0327736,5.3813e-07)
        };
        \node [fill=yellow,draw=black,anchor=center,font=\tiny] at (0.185395,0.000119251) {$3.23$};
        \node [fill=yellow,draw=black,anchor=center,font=\tiny] at (0.0926977,1.32186e-05) {$3.12$};
        \node [fill=yellow,draw=black,anchor=center,font=\tiny] at (0.0463489,1.55397e-06) {$3.06$};
        \addlegendentry{$k = 3$}
    \end{loglogaxis}
    \end{tikzpicture}
    
    &
    
    \begin{tikzpicture}[baseline, trim axis right]
      \begin{loglogaxis}
	[ mark size=4pt, grid=major, small,
	  xlabel={Average mesh size $\overline h$},
        xtick=data,
        xticklabel={\pgfmathparse{exp(\tick)}\pgfmathprintnumber{\pgfmathresult}},
        xticklabel style={
        /pgf/number format/.cd,fixed,precision=2 },
	title={$L^2$ errors ($e_0$)}, width=0.5\textwidth ]
        \addplot[color=red,mark=x] coordinates {
          (0.262189,0.00525006)
          (0.131094,0.00122044)
          (0.0655472,0.000288205)
          (0.0327736,6.86082e-05)
        };
        \node [fill=yellow,draw=black,anchor=center,font=\tiny] at (0.185395,0.00253129) {$2.10$};
        \node [fill=yellow,draw=black,anchor=center,font=\tiny] at (0.0926977,0.000593075) {$2.08$};
        \node [fill=yellow,draw=black,anchor=center,font=\tiny] at (0.0463489,0.000140617) {$2.07$};
        \addplot[color=green,mark=+] coordinates {
          (0.262189,0.000625586)
          (0.131094,6.71301e-05)
          (0.0655472,7.76154e-06)
          (0.0327736,9.34927e-07)
        };
        \node [fill=yellow,draw=black,anchor=center,font=\tiny] at (0.185395,0.000204928) {$3.22$};
        \node [fill=yellow,draw=black,anchor=center,font=\tiny] at (0.0926977,2.28262e-05) {$3.11$};
        \node [fill=yellow,draw=black,anchor=center,font=\tiny] at (0.0463489,2.69379e-06) {$3.05$};
        \addplot[color=blue,mark=o] coordinates {
          (0.262189,2.09333e-05)
          (0.131094,1.01546e-06)
          (0.0655472,5.54436e-08)
          (0.0327736,3.23771e-09)
        };
        \node [fill=yellow,draw=black,anchor=center,font=\tiny] at (0.185395,4.61052e-06) {$4.37$};
        \node [fill=yellow,draw=black,anchor=center,font=\tiny] at (0.0926977,2.37278e-07) {$4.19$};
        \node [fill=yellow,draw=black,anchor=center,font=\tiny] at (0.0463489,1.33982e-08) {$4.10$};
    \end{loglogaxis}
    \end{tikzpicture}
  \end{tabular}
\ref{test3d-bubble-H1err-2}
  \caption{Convergence plots for Test~\ref{test3d-bubble-convergence}, case $\beta = 2$.}
  \label{fig:test3d-bubble-convergence-2}
\end{figure}
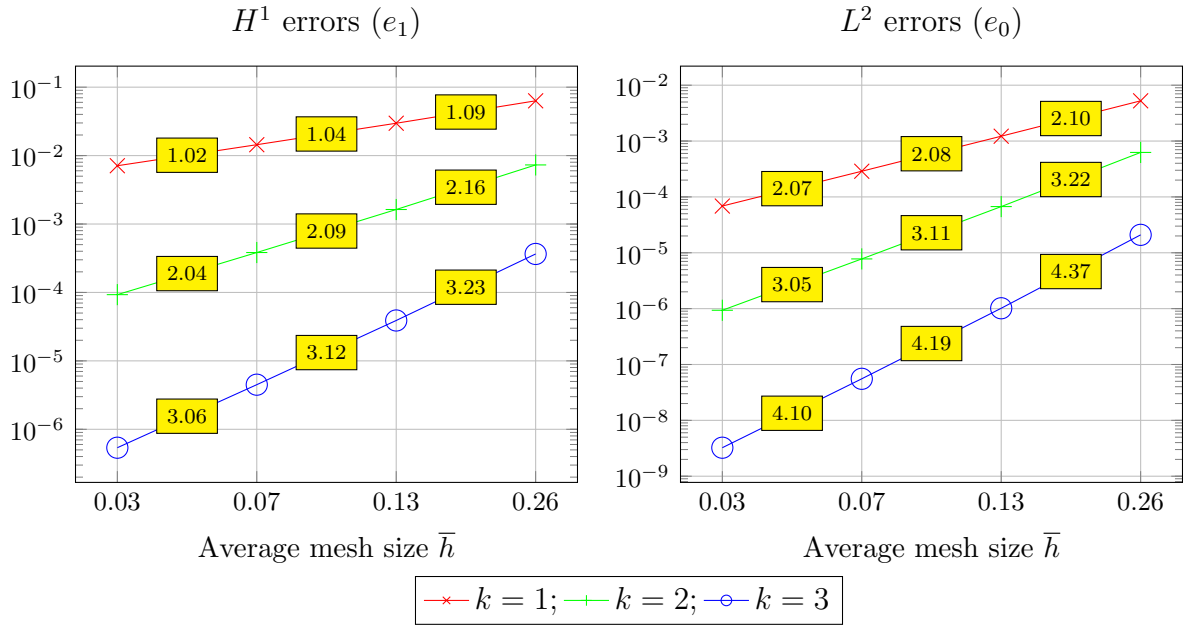

\bibliographystyle{amsplain}

\bibliography{biblio}

\providecommand{\bysame}{\leavevmode\hbox to3em{\hrulefill}\thinspace}
\providecommand{\MR}{\relax\ifhmode\unskip\space\fi MR }
\providecommand{\MRhref}[2]{%
  \href{http://www.ams.org/mathscinet-getitem?mr=#1}{#2}
}
\providecommand{\href}[2]{#2}
\begin{thebibliography}{10}

\bibitem{ahmad2013}
B.~Ahmad, A.~Alsaedi, F.~Brezzi, L.~D. Marini, and A.~Russo, \emph{Equivalent
  projectors for virtual element methods}, Computers and Mathematics with
  Applications \textbf{66} (2013), no.~3, 376--391.

\bibitem{anand2020}
A.~Anand, J.~S. Ovall, S.~E. Reynolds, and S.~Wei\ss{}er, \emph{Trefftz finite
  elements on curvilinear polygons}, SIAM Journal on Scientific Computing
  \textbf{42} (2020), no.~2, A1289--A1316.

\bibitem{antolin2022}
P.~Antolin, X.~Wei, and A.~Buffa, \emph{Robust numerical integration on curved
  polyhedra based on folded decompositions}, Computer Methods in Applied
  Mechanics and Engineering \textbf{395} (2022), 114948.

\bibitem{atallah2020}
N.~M. Atallah, C.~Canuto, and G.~Scovazzi, \emph{The second-generation {Shifted
  Boundary Method} and its numerical analysis}, Computer Methods in Applied
  Mechanics and Engineering \textbf{372} (2020), 113341.

\bibitem{nonconforming}
B.~Ayuso, K.~Lipnikov, and G.~Manzini, \emph{The nonconforming virtual element
  method}, ESAIM Math. Model. Numer. Anal. \textbf{50} (2016), no.~3, 879--904.

\bibitem{basic_vem}
L.~Beir\~{a}o~da Veiga, F.~Brezzi, A.~Cangiani, G.~Manzini, L.~D. Marini, and
  A.~Russo, \emph{Basic principles of virtual element methods}, Mathematical
  Models and Methods in Applied Sciences \textbf{23} (2013), no.~01, 199--214.

\bibitem{beirao2020}
L.~Beir\~{a}o~da Veiga, F.~Brezzi, L.~D. Marini, and A.~Russo, \emph{Polynomial
  preserving virtual elements with curved edges}, Mathematical Models and
  Methods in Applied Sciences \textbf{30} (2020), no.~08, 1555--1590.

\bibitem{MFD}
L.~Beir\~{a}o~da Veiga, K.~Lipnikov, and G.~Manzini, \emph{{The Mimetic Finite
  Difference Method for Elliptic Problems}}, Springer Cham, 2013.

\bibitem{beirao2017}
L.~Beir\~{a}o~da Veiga, C.~Lovadina, and A.~Russo, \emph{Stability analysis for
  the virtual element method}, Mathematical Models and Methods in Applied
  Sciences \textbf{27} (2017), no.~13, 2557--2594.

\bibitem{beirao_curvo_2019}
L.~Beir\~{a}o~da Veiga, A.~Russo, and G.~Vacca, \emph{The virtual element
  method with curved edges}, ESAIM:M2AN \textbf{53} (2019), no.~2, 375--404.

\bibitem{max3}
L.~Beir{\~a}o~{da Veiga}, F.~Brezzi, F.~Dassi, L.~D. Marini, and A.~Russo,
  \emph{A family of three-dimensional virtual elements with applications to
  magnetostatics}, SIAM J. Numer. Anal. \textbf{56} (2018), no.~5, 2940--2962.

\bibitem{BBDMR-Cina}
\bysame, \emph{Serendipity virtual elements for general elliptic equations in
  three dimensions}, Chinese Annals of Mathematics Series B \textbf{39} (2018),
  no.~2, 315--334.

\bibitem{Acta-VEM}
L.~Beir{\~a}o~{da Veiga}, F.~Brezzi, L.~D. Marini, and A.~Russo, \emph{The
  virtual element method}, Acta Numerica \textbf{32} (2023), 123--202.

\bibitem{bertoluzza2024}
S.~Bertoluzza, M.~Montardini, M.~Pennacchio, and D.~Prada, \emph{The virtual
  element method on polygonal pixel--based tessellations}, Journal of
  Computational Physics \textbf{518} (2024), 113334.

\bibitem{bertoluzza2022}
S.~Bertoluzza, M.~Pennacchio, and D.~Prada, \emph{Weakly imposed dirichlet
  boundary conditions for 2d and 3d virtual elements}, Computer Methods in
  Applied Mechanics and Engineering \textbf{400} (2022), 115454.

\bibitem{HHO-curved}
L.~Botti and D.~A. Di~Pietro, \emph{{A}ssessment of {H}ybrid {H}igh-{O}rder
  methods on curved meshes and comparison with discontinuous {G}alerkin
  methods}, J. Comput. Phys. \textbf{370} (2018), 58--84.

\bibitem{brenner2017}
S.~C. Brenner, Q.~Guan, and Li-Yeng Sung, \emph{Some estimates for virtual
  element methods}, Computational Methods in Applied Mathematics \textbf{17}
  (2017), no.~4, 553--574.

\bibitem{brenner2018}
S.~C. Brenner and Li-Yeng Sung, \emph{Virtual element methods on meshes with
  small edges or faces}, Mathematical Models and Methods in Applied Sciences
  \textbf{28} (2018), no.~7, 1291--1336.

\bibitem{curved_bm}
F.~Brezzi and L.~D. Marini, \emph{Virtual elements on polyhedra with a curved
  face}, Bulletin of Mathematical Sciences \textbf{13} (2023), no.~3, 2350005.

\bibitem{chin2021}
E.~B. Chin and N.~Sukumar, \emph{Scaled boundary cubature scheme for numerical
  integration over planar regions with affine and curved boundaries}, Computer
  Methods in Applied Mechanics and Engineering \textbf{380} (2021), 113796.

\bibitem{Wriggers-1}
P.~Wriggers, W.T. Rust, and B.D. Reddy, \emph{A virtual element method for
  contact}, Comput. Mech. \textbf{58} (2016), no.~6, 1039--1050.

\end{thebibliography}







\end{document}